\documentclass{amsart}
\usepackage{amsmath,amsthm,amsopn,amstext,amscd,amsfonts,amssymb,mathrsfs,mathtools, esint}
\usepackage{dsfont}
\usepackage{comment}
\usepackage{graphicx}

\usepackage{epsfig}
\usepackage{psfrag}

\usepackage{infix-RPN}
\usepackage{pst-plot,pst-infixplot,pstricks,graphicx}
\usepackage{amssymb,latexsym,amsthm,amsfonts,color,fancyhdr}

\newrgbcolor{verde}{0.0 0.5 0.0}

\makeatletter
\DeclareFontFamily{U}{tipa}{}
\DeclareFontShape{U}{tipa}{m}{n}{<->tipa10}{}
\newcommand{\arc@char}{{\usefont{U}{tipa}{m}{n}\symbol{62}}}%

\newcommand{\arc}[1]{\mathpalette\arc@arc{#1}}

\newcommand{\arc@arc}[2]{%
  \sbox0{$\m@th#1#2$}%
  \vbox{
    \hbox{\resizebox{\wd0}{\height}{\arc@char}}
    \nointerlineskip
    \box0
  }%
}
\makeatother

\newmuskip\pFqskip
\mathchardef\pFcomma=\mathcode`, 

\newcommand*\pFq[5]{%
 \begingroup
 \begingroup\lccode`~=`,
   \lowercase{\endgroup\def~}{\pFcomma\mkern\pFqskip}%
 \mathcode`,=\string"8000
 {}_{#1}F_{#2}\left(\genfrac..{0pt}{}{#3}{#4};#5\right)%
 \endgroup
}

\usepackage{float}
\usepackage[active]{srcltx}
\usepackage{graphicx, mathdots}

  \DeclareMathOperator{\diag}{diag}
  
  \usepackage{mathtools}

\DeclarePairedDelimiter\floor{\lfloor}{\rfloor}

\newtheorem{theorem}{\sc Theorem}[section]

\newtheorem{coro}{\sc Corollary}[section]
\newtheorem{proposition}{\sc Proposition}[section]
\newtheorem{obs}{\sc Remark}[section]
\usepackage[active]{srcltx}
\newcommand{\Ima}{{\rm Im}}

\newcommand{\dps}{\displaystyle}
\usepackage{graphicx, epsfig, subfig}
\usepackage{lscape}
\usepackage{rotating}
\usepackage{caption}
\usepackage{multicol}
\usepackage{colortbl}

\begin{document}
\title[Markov theorem for weight functions on the unit circle]{Markov theorem for weight functions on the unit circle}
\author{K. Castillo}
\address{CMUC, Department of Mathematics, University of Coimbra, 3001-501 Coimbra, Portugal}
\email{kenier@mat.uc.pt}

\date{\today}
\begin{abstract}
The aim of this paper is to prove that Markov's theorem on variation of zeros of orthogonal polynomials on the real line [Math. Ann., 27:177--182,
1886] remains essentially valid in the case of paraorthogonal polynomials on the unit circle.\end{abstract}
\maketitle
\section{Introduction}

In 1886, A. A. Markov proved a remarkable theorem concerning the dependence of the zeros of the elements of a sequence of orthogonal polynomials $(p_n)_{n=1}^\infty$ on a real parameter $t$ which appears in the weight function $\omega$ defined on the real interval $[a,b]$ (see \cite[p. $178$]{M86}).  Szeg\H{o} devotes two sections of his classical book to expose Markov's work (see \cite[Sections 6.12 and 6.21]{S75}) and, in a more recent monograph on the subject, Ismail refers Markov's theorem as ``an extremely useful theorem" (see \cite[p. 203]{I05}). The beauty and wide applicability of this result rest on its powerful simplicity: 
\vspace{2mm}
\begin{changemargin}{0.8cm}{0.8cm} 
{\em Under suitable conditions, the zeros of $p_n(\cdot; t)$ are increasing functions of $t$ provided that
$$
\frac{1}{\omega(x; t)}\frac{\partial \omega}{\partial t}(x; t)
$$
is an increasing function of $x$ on $(a,b)$.} 
\end{changemargin}
\vspace{2mm}
As a direct consequence of his result, Markov himself showed that the zeros of Jacobi polynomials, with weight function $\omega(x; \alpha, \beta)=(1-x)^\alpha(1+x)^\beta$ on $[-1,1]$ for $\alpha, \beta \in (-1, \infty)$, are decreasing functions of $\alpha$ and increasing functions of $\beta$. Indeed,
\begin{align*}
\frac{1}{\omega(x; \alpha,\beta)}\frac{\partial \omega(x;\alpha,\beta)}{\partial \alpha}&=\log(1-x),\\
\frac{1}{\omega(x;\alpha,\beta)}\frac{\partial \omega(x;\alpha,\beta)}{\partial \beta}&=\log(1+x).
\end{align*}
Markov also attempts a general theorem to deal with the ultraspherical case $\alpha=\beta$, but his proof is incorrect. A proof of Markov's theorem for even weight functions on $[-1,1]$ ---easy once you realize that mapping $(-1,1)$ into $(0, 1)$ the problem is reduced to the known case--- can be found in \cite[Corollary 2]{KP87}  in a more general context.  

Over the years there were many extensions to the classical theory of orthogonal polynomials on the real line (OPRL). After the influential works by Delsarte and Genin \cite{DG88, DG91a, DG91b} and Jones et al. \cite{JNT89} about the nowadays called paraorthogonal polynomials on the unit circle (POPUC) ---in many senses the appropriate complex analog of OPRL---, this collection of polynomials and their zeros have received considerable attention from two disparate audiences, namely researchers in orthogonal polynomials and researchers in numerical linear algebra (see for instance \cite{G82, G82o, G86, AGR88, DG88, GR90, DG91a, DG91b, W93, B93, GH95, BH95, S05I, KN07, S07, S07b, W07, S11, S12, S16, CCP16, CP18, MSS19, C19, S19, BS19}). It must be said that rarely in the numerical linear algebra context the name POPUC is used; however, the reader has to proceed with caution in the literature because many results on POPUC were first discovered in this framework.  As we will see below, POPUC are closely related with orthogonal polynomials on the unit circle (OPUC) and, therefore, with weight functions on the unit circle. But unfortunately Markov's theorem can not deal with it. 
In Section \ref{main} we discusses this question and investigates the extent to which Markov's theorem remains valid in the case of weight functions on the unit circle. Unlike what happens in the case of OPRL (see the proof of \cite[Theorem $6.12.1.$]{S75} and the hint of \cite[Problem $15$, Chapter III]{G71}), we can not use quadrature for our purpose because Szeg\H{o} quadrature is much weaker than Gaussian quadrature. In Section \ref{examples} we apply our results to some specific families of polynomials, but first some preliminary definitions and basic results are needed (see \cite{S05I, S11} for more details). 
\section{Preliminaries}\label{pre}
Let $\mathrm{d}\mu(\theta)$ be a finite nonnegative measure with infinite support on the unit circle parametrized by $z=e^{i\theta}$ and
$$
c_j=\int e^{-i j \theta}\mathrm{d}\mu(\theta)\quad (j=0,1,2,\dots)
$$
its moments. We will use $c_j(\mathrm{d}\mu)$ if we want the $\mathrm{d}\mu$ dependence to be explicit. Let $(Q_n)_{n=0}^\infty$ be the unique sequence of monic OPUC associated with $\mathrm{d}\mu$, that is, polynomials $Q_n(z; \mathrm{d}\mu)=Q_n(z)=z^n+\cdots$ which satisfy
\begin{align*}
 \int Q_n(e^{i \theta})\overline{Q_m(e^{i \theta})}\,\mathrm{d}\mu(\theta)&=0 \quad (n\not=m=0,1,2,\dots),\\
\int |Q_n(e^{i \theta})|^2\mathrm{d}\mu(\theta)&\not=0.
\end{align*}
Define $c_j$ for $j=-1,-2,-3,\dots$ by $c_j=\overline{c_{-j}}$. We mention the following explicit representation of $Q_n$ sometimes called Heine's formula:
\begin{align}\label{heine}
Q_n(z)=\mathrm{D}_{n-1}(\mathrm{d}\mu)^{-1}\det
\begin{pmatrix}
c_0 & c_{-1} & \cdots & c_{-n}\\
c_1 & c_0 & \cdots & c_{-n+1}\\
\vdots & \vdots & & \vdots\\
c_{n-1} & c_{n-2} & \cdots & c_{-1}\\
1 & z & \cdots & z^n
\end{pmatrix} \quad (n=1,2,\dots)
\end{align}
where $\mathrm{D}_{n-1}(\mathrm{d}\mu)=\det(c_{k-j})_{j,k=0}^{n-1}>0$ by the Carath\'eodory-Toeplitz theorem. Define the normalized OPUC by $q_n(z)=\kappa_n\, z^n+\cdots$ where $\kappa_n=\|Q_n\|^{-1}$. The CD kernel is defined for $w, z\in \mathbb{C}$ by
$$
K_{n}(w, z; \mathrm{d}\mu)=K_{n}(w, z)=\sum_{j=0}^n \overline{q_j(w)}q_j(z).
$$
For any polynomial $f$ of degree at most $n$, we have
\begin{align}\label{CD}
\int  f(e^{i \theta})K_{n}(e^{i\theta}, w) \mathrm{d}\mu(\theta)=f(w),
\end{align}
often called the reproducing property.  

Denote by $\mathbb{S}^1_r(c)$ the boundary of the open disk $\mathbb{D}_r(c)$ of radius $r>0$ with center $c$. Since the unit disk with center at the origin plays a distinguished role in the theory of OPUC, we use the notation  $\mathbb{D}=\mathbb{D}_1(0)$ and $\mathbb{S}^1=\mathbb{S}^1_1(0)$. Fix $n\in \{1,2,\dots\}$ and $b \in \mathbb{S}^1$. The monic POPUC of degree $n$ associated with $\mathrm{d}\mu$ and $b$ is defined by (see \cite[p. $115$]{S11})
\begin{align}\label{popuc}
P_n(z; b; \mathrm{d}\mu)=P_n(z)=zQ_{n-1}(z)-\overline{b} \,Q^*_{n-1}(z),
\end{align}
where $Q_n^*(z)=z^{n} \overline{Q(1/\overline{z})}$. The normalized POPUC is given by $p_n(z; b; \mathrm{d}\mu)=p_n(z)=zq_{n-1}(z)-\overline{b} \,q^*_{n-1}(z)$.   
Another appropriate denomination for POPUC is quasi-orthogonal polynomials on the unit circle, in part because 
\begin{align}\label{exp}
\int  P_n(e^{i\theta}) \overline{g(e^{i\theta})} \,\mathrm{d}\mu(\theta)=0,
\end{align}
 for any polynomial $g$ of degree at most $n-1$ vanishing at the origin, and in part because, as Geronimus pointed out (see \cite[Footnote 10, p. 12]{G48}\footnote{See also \cite[Remark I]{G46}.}), ``this property is analogous to a fundamental property of the so-called quasi-orthogonal polynomials of M. Riesz". The `quasi-orthogonality' condition \eqref{exp} gives rise to some interesting properties of POPUC. Suppose that $P_n(\zeta)=0$ and let $h$ be a nonzero polynomial of degree at most $n-1$. Since $h(z)-h(\zeta)$ has a zero of multiplicity at least one at $z=\zeta$, 
$$
\frac{z h(z)-z h(\zeta)}{z-\zeta}
$$
is a polynomial of degree $n-1$ vanishing at the origin. From \eqref{exp}, we have\footnote{As we will see later $\zeta \in \mathbb{S}^1$, and so $\zeta\not=0$.}
$$
0=-\frac{1}{\zeta} \int  P_n(e^{i \theta})\, \frac{\overline{h(e^{i \theta})-h(\zeta)}}{e^{i \theta} \overline{(e^{i \theta}-\zeta)}} \mathrm{d}\mu(\theta)=\int  \frac{P_n(e^{i \theta})}{e^{i \theta}-\zeta} \, \overline{(h(e^{i \theta})-h(\zeta))}\mathrm{d}\mu(\theta).
$$
Hence, 
 \begin{align}\label{property}
\int  \frac{P_n(e^{i \theta})}{e^{i \theta}-\zeta}\,\, \overline{h(e^{i\theta})}\mathrm{d}\mu(\theta)=\overline{h(\zeta)} \int  \frac{P_n(e^{i \theta})}{e^{i \theta}-\zeta} \mathrm{d}\mu(\theta),
\end{align}
for any polynomial $h$ of degree at most $n-1$. Moreover, since there exists $C\in \mathbb{C}\setminus \{0\}$ (cf. \cite[p. 284]{W07}) such that
 \begin{align*}
 P_n(z)=C (z-\zeta)K_{n-1}(\zeta, z),
 \end{align*} 
 \eqref{CD} shows that
 \begin{align}\label{C}
C=\int  \frac{P_n(e^{i \theta})}{e^{i \theta}-\zeta}\mathrm{d}\mu(\theta)\not=0.
 \end{align}

Denote by $a_j=-\overline{Q_{j+1}(0)}$  the Verblunsky coefficients. Set
\begin{align*}
\Theta_j=\Theta(a_j)=\begin{pmatrix}
\overline{a}_j & \ \ r _j\\
r_j & -a_j
\end{pmatrix},
\end{align*}
where $r_j=\left(1-|a_j|^2\right)^{1/2}$.
Define $\mathrm{G}_j=\diag \left(\mathrm{I}_{j}, \Theta_j, \mathrm{I}_{n-j-2} \right)$ and $\mathrm{G}_{n-1}=\diag(\mathrm{I}_{n-1}, \overline{b})$. (Here $\mathrm{I}$ denotes the identity matrix, whose order is made explicit with a subindex.) It is well known that $P_n$ is the characteristic polynomial of the GGT unitary matrix (see for instance \cite[(4.19)]{DG91b})
\begin{align}\label{G}
\mathrm{G}=\mathrm{G}_0\mathrm{G}_1\cdots \mathrm{G}_{n-1}.
\end{align}
 In practical work it is not always necessary to write this matrix explicitly, but it is important to known that $\mathrm{G}$ is a unitary upper Hessenberg matrix with positive subdiagonal elements. Therefore the zeros of POPUC have two very attractive properties: (1) All the zeros of $P_n$ lie on $\mathbb{S}^1$; (2) The zeros of $P_n$ are all simple (see a different proof in \cite[Theorem $9.1.$]{G48}).
\section{Main results}\label{main}
Let us introduce the notation $C_r(c)=\mathbb{D}_r(c)\cap \mathbb{S}^1$ and $I_r(c)=\mathbb{D}_r(c)\cap \mathbb{R}$. In what follows we shall use (explicitly or implicitly) the following result. 

\begin{proposition}\label{pmain}
Let $\mathrm{d}\mu(\theta; t)=\omega(\theta; t)\, \mathrm{d}\mu(\theta)$ be a finite nonnegative measure with infinite support on the unit circle parametrized by $z=e^{i\theta}$ $(\theta \in [\theta_0, \theta_0+2\pi))$ and depending on a parameter $t$ varying  in a real open interval containing $t_0$. Suppose that for almost all $\theta \in [\theta_0, \theta_0+2\pi)$,
$\omega(\theta; t)$ is finite and admits partial derivative with respect to $t$. Suppose furthermore that there exists a $\mu$-integrable function $\alpha$ such that $$\left|\frac{\partial\omega}{\partial t}(\theta; t)\right|\leq \alpha(\theta),$$
almost everywhere in $ [\theta_0, \theta_0+2\pi)$. Let $P(z; t)$ be a nonconstant monic POPUC associated with $\mathrm{d}\mu(\theta; t)$.  Assume that $P(\zeta_0; t_0)=0$.  Then there exist $\epsilon>0$ and $\delta>0$ such that $C_\delta(\zeta_0) \times I_\epsilon(t_0)$ is in the neighbourhood where $P$ is defined, and there exists $\zeta: I_\epsilon(t_0) \to C_\delta(\zeta_0)$, such that 
\begin{align}\label{zero}
P(\zeta(t); t)=0
\end{align}
and, for each $t \in I_\epsilon(t_0)$, $\zeta$ is the unique solution of \eqref{zero} with $\zeta(t)\in C_\delta(\zeta_0)$. Moreover, $\zeta$ possess continuous derivatives on $I_\epsilon(t_0)$.
\end{proposition}

\begin{proof}
Assume that $P$ has fixed positive degree $n$. From \eqref{heine} we see that the coefficients of $P$ are rational functions of $c_j(\mathrm{d}\mu)$ $(j=-n, \dots, n-2,n-1)$, where the denominator is the determinant $\mathrm{D}_{n-1}(\mathrm{d}\mu)$. Under our hypotheses, we can differentiate 
$$
c_j(\mathrm{d}\mu(\cdot; t))=\int e^{-i j \theta}\omega(\theta; t) \mathrm{d}\mu(\theta)
$$
under the integral sign (cf. \cite[pp. 124-125]{D70}); we see immediately then that the coefficients of $P(\cdot; t)$ are differentiable functions for each $t$. Moreover, $P(\zeta_0; t_0)=0$; from this it follows that 
$$
\left. \frac{\partial P}{\partial z}(z; t)\right|_{z=\zeta_0, t=t_0}\not=0,
$$
and the result is a direct consequence of the analytic implicit function theorem (see \cite[Theorem 3.4.2]{Si15}).
\end{proof}

We shall refer to Theorem \ref{markov} below as circular Markov theorem with a fixed zero.

\begin{theorem}\label{markov}
 Assume the hypotheses and notation of Proposition \ref{pmain}. Assume also that $P(e^{i\theta_0}; t)=0$ for each $t\in  I_\epsilon(t_0)$. Suppose that $\omega(\theta; t)$ is positive and continuous for each $\theta \in [\theta_0, \theta_0+2\pi)$ and $t \in I_\epsilon(t_0)$.  Suppose furthermore that the partial derivative of $\omega(\theta; t)$ with respect to $t$ is continuous for each $\theta \in [\theta_0, \theta_0+2\pi)$ and $t \in I_\epsilon(t_0)$. Then $\zeta(t)$ moves strictly counterclockwise along $\mathbb{S}^1$ as $t$ increases on $I_\epsilon(t_0)$, provided that
 \begin{align}\label{mark}
\frac{1}{\omega(\theta; t)}\frac{\partial \omega}{\partial t}(\theta; t)
\end{align}
is a strictly increasing function of $\theta$ on  $(\theta_0, \theta_0+2\pi)$.
\end{theorem}
\begin{proof}
Assume that $P$ has fixed degree $n\geq 2$ and write $P_n$ instead of $P$. By the analytic implicit function theorem, we have
\begin{align}\label{derivative}
 \zeta'(t)=\dps -\frac{\dps \frac{\partial P_n}{\partial t}(\zeta(t); t)}{\dps \frac{\partial P_n}{\partial z} (\zeta(t); t)}
\end{align}
for each $t \in I_\epsilon(t_0)$.
Since the leading coefficient of $P_n(\cdot; t)$ does not depend on $t$, \eqref{property} and \eqref{C} make it obvious that
\begin{align}
\label{part2} 
\frac{\partial P_n}{\partial t}(\zeta(t); t) =\frac{\dps \int  \frac{\overline{P_n(e^{i \theta}; t)}}{\overline{e^{i \theta}-\zeta(t)}} \dps \frac{\partial P_n}{\partial t}(e^{i\theta}; t) \mathrm{d}\mu(\theta; t)}{\dps \int  \frac{\overline{P_n(e^{i \theta}; t)}}{\overline{e^{i \theta}-\zeta(t)}}\mathrm{d}\mu(\theta; t)}.
\end{align}
Define the polynomial of degree $n$ in $z$,
$$
R(z; t)=P_n(z; t)-\frac{\partial P_n}{\partial z}(\zeta(t); t) (z-\zeta(t)).
$$
Since $R(z; t)$ has a zero of multiplicity at least two at $z=\zeta(t)$, 
$$
\frac{z R(z; t)}{(z-\zeta(t))^2}
$$
is a nonzero polynomial of degree $n-1$ in $z$ vanishing at the origin. Therefore
 \begin{align}
\nonumber 0&=-\frac{1}{\zeta(t)}\int  P_n(e^{i\theta}; t)\, \frac{\overline{R(e^{i\theta}; t)}}{e^{i\theta}\overline{(e^{i\theta}-\zeta(t))}^2}\mathrm{d}\mu(\theta; t)\\[7pt]
\label{part1} &=\int  \left|\frac{P_n(e^{i \theta}; t)}{e^{i \theta}-\zeta(t)}\right|^2\mathrm{d}\mu(\theta; t)-\overline{\frac{\partial P_n}{\partial z}(\zeta(t); t)} \int  \frac{P_n(e^{i \theta}; t)}{e^{i \theta}-\zeta(t)}\mathrm{d}\mu(\theta; t)
\end{align}
by \eqref{exp}.
Combining \eqref{part2} with \eqref{part1}  we can rewrite \eqref{derivative} as
\begin{align}\label{derivative1}
 \zeta'(t)=-\dps \frac{\dps \int  \frac{\overline{P_n(e^{i \theta}; t)}}{\overline{e^{i \theta}-\zeta(t)}} \dps \frac{\partial P_n}{\partial t}(e^{i\theta}; t)\mathrm{d}\mu(\theta; t)}{\dps \int  \left|\frac{P_n(e^{i \theta}; t)}{e^{i \theta}-\zeta(t)}\right|^2\mathrm{d}\mu(\theta; t)}.
\end{align}
Write $\xi =e^{i\theta_0}$. From \eqref{exp}, we also get
\begin{align}\label{ort0}
0=\int  \frac{\overline{P_n(e^{i \theta}; t)}}{\overline{e^{i \theta}-\xi}} \dps \frac{\partial P_n}{\partial t}(e^{i\theta}; t)\mathrm{d}\mu(\theta; t).
\end{align}
Write $\zeta(t)=e^{i\varphi(t)}$ $(\varphi(t) \in [\theta_0, \theta_0+2\pi))$ and let $C(t)$ denotes the denominator of the right hand side of \eqref{derivative1}. Note that 
$$
\frac{i\xi}{e^{i\theta}-\xi}-\frac{i\zeta(t)}{e^{i\theta}-\zeta(t)}=\frac{i (\xi-\zeta(t)) e^{i\theta}}{(e^{i\theta}-\xi)(e^{i\theta}-\zeta(t))}.
$$
 If \eqref{derivative1} and \eqref{ort0} are multiplied by $-i\overline{\zeta(t)}$ and $-i\overline{\xi}$ respectively and the resulting equations are added, we have
\begin{align}\label{newder}
C(t)\varphi'(t)=\int \frac{i (\zeta(t)-\xi) e^{i\theta}}{(e^{i\theta}-\xi)(e^{i\theta}-\zeta(t))} P_n(e^{i\theta}; t) \overline{\frac{\partial P_n}{\partial t}(e^{i\theta}; t)}\mathrm{d}\mu(\theta; t).
\end{align}
Since 
$$
\frac{z P_n(z, t)}{(z-\xi)(z-\zeta(t))}
$$ 
is a nonzero polynomial of degree $n-1$ in $z$ vanishing at the origin, \eqref{exp} yields 
\begin{align}
 \label{exp0} 0=\int  \frac{e^{i\theta}}{(e^{i\theta}-\xi)(e^{i\theta}-\zeta(t))}  |P_n(e^{i\theta}; t)|^2\,  \mathrm{d}\mu(\theta; t).
\end{align}
Taking the partial derivative of \eqref{exp0} with respect to $t$ and using \eqref{exp} leads to
\begin{align}
\label{aux0} &\int \frac{e^{i\theta}}{(e^{i\theta}-\xi)(e^{i\theta}-\zeta(t))} P_n(e^{i\theta}; t) \overline{\frac{\partial P_n}{\partial t}(e^{i\theta}; t)}\mathrm{d}\mu(\theta; t)\\[7pt]
\nonumber &=-\int  \frac{e^{i\theta}}{(e^{i\theta}-\xi)(e^{i\theta}-\zeta(t))}  |P_n(e^{i\theta}; t)|^2\,\frac{\partial \omega}{\partial t}(\theta; t) \mathrm{d}\mu(\theta).
\end{align}
Define the real-valued function
\begin{align*}
\varpi(\theta; t)=\dps\frac{1}{\dps \omega(\theta; t)}\frac{\partial \omega}{\partial t}(\theta; t)- \frac{1}{\omega(\varphi(t); t)}\frac{\partial \omega}{\partial t}(\varphi(t); t).
\end{align*}
Combining \eqref{exp0} with \eqref{aux0} we deduce that
\begin{align}
\label{last0}&-\int \frac{e^{i\theta}}{(e^{i\theta}-\xi)(e^{i\theta}-\zeta(t))} P_n(e^{i\theta}; t) \overline{\frac{\partial P_n}{\partial t}(e^{i\theta}; t)}\mathrm{d}\mu(\theta; t)\\[7pt]
\nonumber &=\int  \frac{e^{i\theta}}{(e^{i\theta}-\xi)(e^{i\theta}-\zeta(t)}  |P_n(e^{i\theta}; t)|^2\,\varpi(\theta; t) \mathrm{d}\mu(\theta; t).
\end{align}
Substituting \eqref{last0} into \eqref{newder}, we can assert that
\begin{align}\label{int}
C(t) \varphi'(t)= \int  \frac{i(\xi-\zeta(t))\,e^{i\theta}}{(e^{i\theta}-\xi)(e^{i\theta}-\zeta(t))}  |P_n(e^{i\theta}; t)|^2\,\varpi(\theta; t) \mathrm{d}\mu(\theta; t).
\end{align}
\begin{figure}[H]
\centering
 \includegraphics[width=8cm]{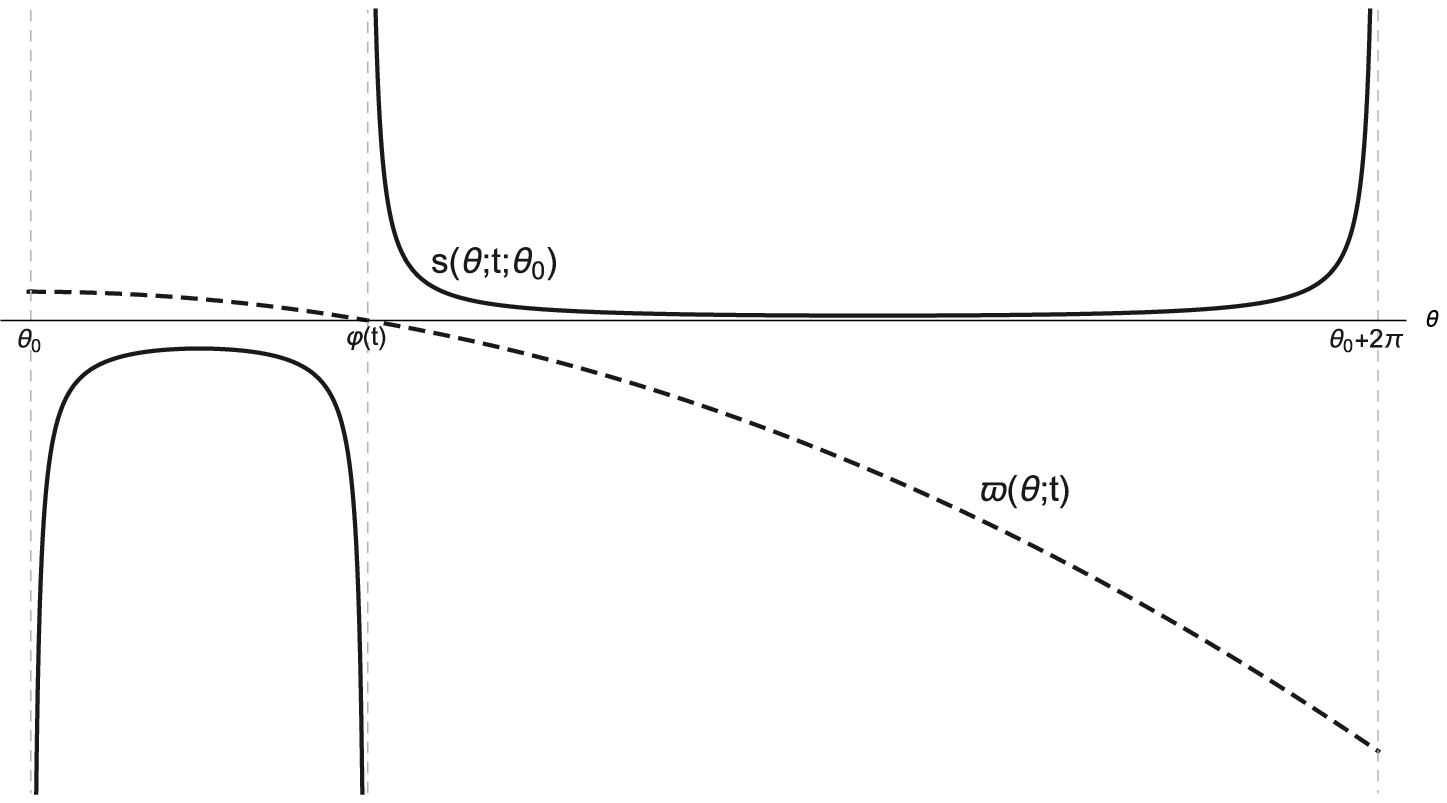}\label{proof1}
 \caption{$s$ and $\varpi$ for the circular Markov theorem with a fixed zero.}
\end{figure} 
Observe  that, for each $t \in I_\epsilon(t_0)$, the real-valued function
 $$
s(\theta; t; \theta_0)= \frac{i(\xi-\zeta(t))\,e^{i\theta}}{(e^{i\theta}-\xi)(e^{i\theta}-\zeta(t))}=\dps-\frac{1}{2} \frac{\sin\left(\dps\frac{\varphi(t)-\theta_0}{2}\right)}{ \sin\left(\dps\frac{\varphi(t)-\theta}{2}\right)\sin\left(\dps\frac{\theta_0-\theta}{2}\right)} 
 $$
 is negative for $\theta \in (\theta_0, \varphi(t))$ and positive for $\theta \in (\varphi(t), \theta_0+2\pi)$. Since, for each $t \in I_\epsilon(t_0)$, $\varpi(\theta; t)$ is positive for $\theta\in (\theta_0, \varphi(t))$ and negative for $\theta \in (\varphi(t), \theta_0+2\pi)$, $\varphi'(t)$ is negative $($see Figure 1$)$, and the theorem is proved.
\end{proof}

\begin{obs}
Theorem \ref{markov} specializes to  \cite[Theorem $3$]{L19} if $\theta_0=0$ $($or what is the same, $P(1; t)=0$$)$  and $\mathrm{d}\mu(\theta)=\mathrm{d}\theta$, the Lebesgue measure. It is important to highlight that unlike \cite{L19}, where several previous results related to the particular case considered are needed, our arguments make use only of the condition \eqref{exp}. We also note that virtually \cite[Theorem $1$]{L19} and the main sentence of \cite[Theorem $2$]{L19} are already proved in \cite[Section $5$]{DG88}\footnote{The reader must recall that the recurrence relation \cite[(1.1)]{L19} can be transformed into the simplest form \cite[(2.12)]{DG88} by a normalization process (see \cite[pp. $226$-$227$]{DG91a}). In any case, \cite[Theorem $1$]{L19} is proved in a more general setting in \cite[Corollary $2.14.5$]{S11} $($see in this regard Remark \ref{remarkPro} below$)$.} and \cite[Theorem B]{C15}\footnote{A refined version of \cite[Theorem B]{C15} can be find in \cite[Corollary $3.2$]{C19}, see also preprint available at arXiv:$1706.05709$ (2017).}, respectively. 
\end{obs}

Even when the integrand of \eqref{int} change sign in the interval of integration $\varphi'$ may have a constant sign in $I_\epsilon(t_0)$. We illustrate this possibility by proving the following result,  which we will use later in Section \ref{examples}. 

\begin{coro}\label{end}
Assume the hypotheses and notation of Theorem \ref{markov} and its proof. Set $\theta_0=-\pi$. Assume that $\mathrm{d}\mu(\theta)=-\mathrm{d}\mu(-\theta)$. Assume also that $\omega(-\theta; t)\geq \omega(\theta; t)$ and 
$$
\frac{1}{\omega(-\theta; t)}\frac{\partial \omega}{\partial t}(-\theta; t)=\frac{1}{\omega(\theta; t)}\frac{\partial \omega}{\partial t}(\theta; t)
$$
for almost all $\theta \in (0, \varphi(t))$ and $t \in I_\epsilon(t_0)$. Suppose that \eqref{mark} is a strictly decreasing function of $\theta$ on  $(-\pi, 0)$ and a strictly increasing function of $\theta$ on  $(0, \pi)$. Then either $\zeta(t)$ moves strictly clockwise along $\mathbb{S}^1$ as $t$ increases on $I_\epsilon(t_0)$ if $\varphi(t)\in (-\pi, 0)$ or else  $\zeta(t)$ moves strictly counterclockwise if $\varphi(t)\in (0, \pi)$.
\end{coro}
\begin{proof}
Set 
$
W(\theta; t)=s(\theta; t; 0)$ $ |P_n(e^{i\theta}; t)|^2\varpi(\theta; t)\omega(\theta; t)$. Suppose that $\varphi(t) \in (0,\pi)$  for each $t \in I_\epsilon(t_0)$. Observe  that $s(\theta; t; 0)$ is positive for each $\theta \in (-\pi, 0) \cup (\varphi(t), \pi)$. Since  $\varpi(\theta; t)$ is positive for each $\theta\in (-\pi,-\varphi(t))\cup (\varphi(t), \pi)$, $W(\theta; t)$ is positive for $\theta\in (-\pi, -\varphi(t))\cup (\varphi(t), \pi)$ (see Figure $2$).  Moreover, 
$$
s(-\theta; t; 0)=-\frac{\sin \left(\dps\frac{\varphi(t)-\theta}{2}\right)}{\sin \left(\dps\frac{\varphi(t)+\theta}{2}\right)}\,s(\theta; t; 0)< -s(\theta; t; 0)
$$
 for each $\theta\in(0, \varphi(t))$. Hence
\begin{align*}
C(t)\varphi'(t)&>\int_{-\varphi(t)}^{\varphi(t)} W(\theta; t) \mathrm{d}\mu(\theta)\\
&=\int_{0}^{\varphi(t)} \big(W(-\theta; t)+W(\theta; t)\big) \mathrm{d}\mu(\theta)>0,
\end{align*}
and so $\varphi'> 0$.  
\begin{figure}[H]
\centering
\includegraphics[width=9cm]{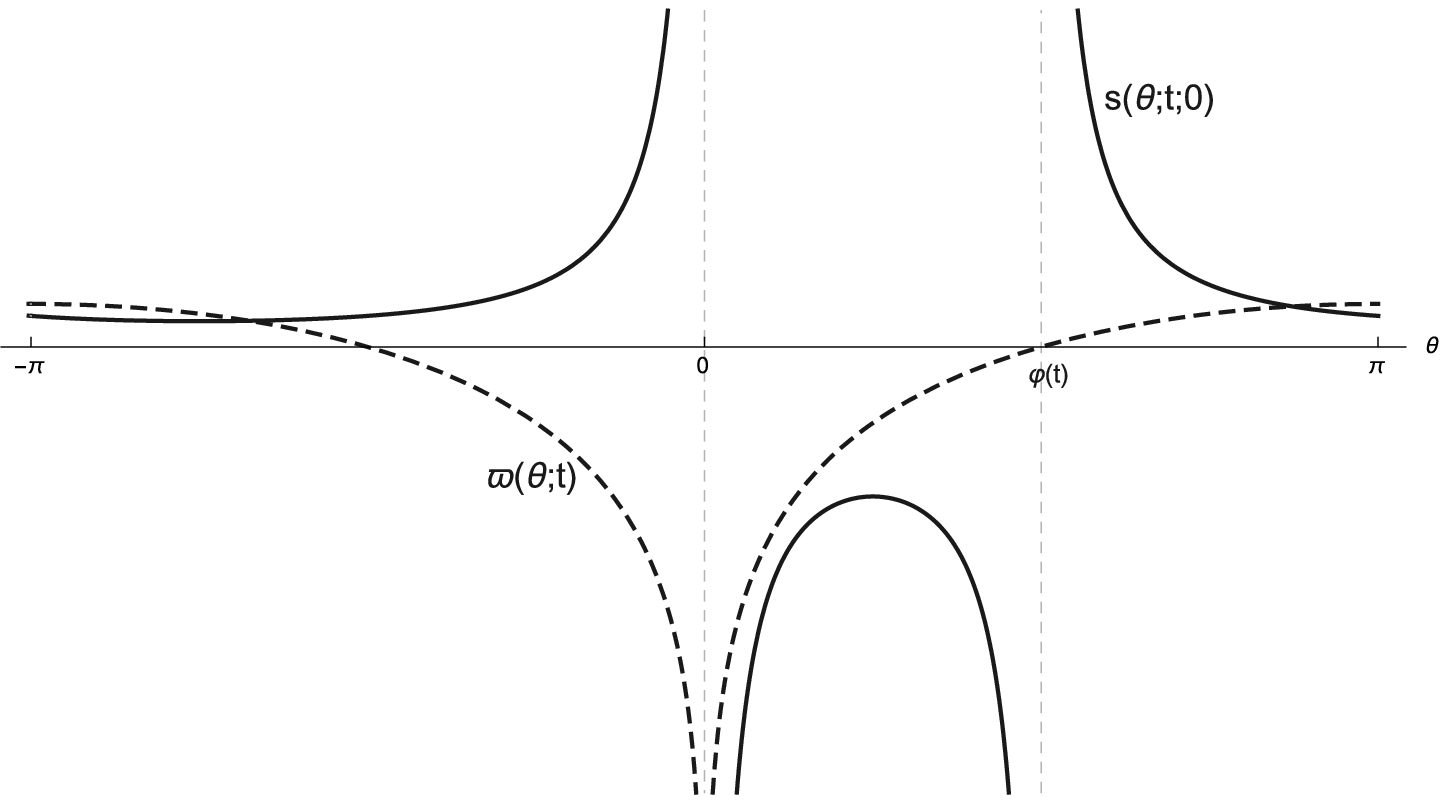}\label{proof2}
\caption{$s$ and $\varpi$ for a consequence of the circular Markov theorem with a fixed zero.}
\end{figure} 
The proof for $\varphi(t) \in (-\pi,0)$ is similar.
\end{proof}

\begin{obs}
We can go even further, however. Note that the result we want to prove is $S=I_1+I_2+I_3+I_4>0$ $($see Figure 2$)$, where
\begin{align*}
I_1&=\int_{-\pi}^{-\varphi(t)} W(\theta; t) \mathrm{d}\mu(\theta)>0, \quad &I_3&=\int_{0}^{\varphi(t)} W(\theta; t) \mathrm{d}\mu(\theta)>0,&\\
I_2&=\int_{-\varphi(t)}^{0} \,\,W(\theta; t) \mathrm{d}\mu(\theta)<0, \quad &I_4&=\int_{\varphi(t)}^{\pi}\,\, W(\theta; t) \mathrm{d}\mu(\theta)>0,&
\end{align*}
and, although under our hypothesis $I_2+I_3>0$, there may be cases in which $I_2+I_3<0$ and still $S>0$.
\end{obs}

\begin{obs}\label{remarkkernel}
POPUC with a fixed zero $($see for instance \cite[(2.11-2.13)]{DG88}$)$ are widely used in practice. This collection of polynomials is closely related with certain CD kernels. Indeed, given $\xi \in \mathbb{S}^1$ and  a measure $\mathrm{d}\mu$ as defined in Section \ref{pre}, the corresponding normalized POPUC of degree $n$ with parameter 
\begin{align}\label{b}
b(\xi)=\overline{\xi}\, \frac{\overline{q_{n-1}(\xi)}}{\overline{q^*_{n-1}(\xi)}}
\end{align}
is given by $($see \cite[(3.7)-(3.8)]{W07} and \cite[Theorem $2.14.3.$ (ii)]{S11}$)$
$$
p_n\left(z; b(\xi); \mathrm{d}\mu \right)=-\frac{\overline{b(\xi)}}{q^*_{n-1}(\xi)} (1-z\overline{\xi})K_{n-1}(\xi, z).
$$
Therefore, the zeros of $K_{n-1}(\xi, \cdot)$ are precisely the zeros of  $p_n\left(\cdot; b(\xi); \mathrm{d}\mu\right)$ other than $\xi$, and the zeros of $p_n\left(\cdot; b(\xi); \mathrm{d}\mu\right)$ are $\xi$ plus the zeros of $K_{n-1}(\xi, \cdot)$. 
\end{obs}

With Theorem \ref{markov} under our belt, the following consequence essentially follows as for the case of OPRL (see \cite[Theorem 6.12.2]{S75}). 

\begin{coro}\label{comp}
Let $\mathrm{d}\mu_1(\theta)=\omega_1(\theta)\, \mathrm{d}\mu(\theta)$ and $\mathrm{d}\mu_2(\theta)=\omega_2(\theta)\, \mathrm{d}\mu(\theta)$  be two nonnegative measures with infinite support on the unit circle parametrized by $z=e^{i\theta}$ $(\theta \in [\theta_0, \theta_0+2\pi))$ and satisfying the hypotheses of Theorem \ref{markov}. Suppose that $\omega_1(\theta)$ and $\omega_2(\theta)$ are finite, positive and continuous for almost all $\theta$. Let $\omega_2(\theta)/\omega_1(\theta)$ be a strictly increasing on  $[\theta_0, \theta_0+2\pi)$. Fix $n\geq2$ and let $\theta_0+2\pi> \theta_{1,1}>\cdots >\theta_{n,1}\geq \theta_0$ and $\theta_0+2\pi> \theta_{1,2}>\cdots >\theta_{n,2}\geq \theta_0$ denote the arguments of the zeros of the POPUC of degree $n$ associated with $\mathrm{d}\mu_1$ and $\mathrm{d}\mu_2$, respectively. Then if $\theta_{k,1}=\theta_{k,2}$ for some $k\in \{1,\dots, n\}$, we have
$$
\theta_{j,1}<\theta_{j,2}
$$ 
for each $j\not=k$.
\end{coro}
\begin{proof}
Define $\mathrm{d}\sigma(\theta; t)=\omega(\theta; t) \mathrm{d}\mu(\theta)$, where $\omega(\theta; t)=(1-t)\omega_1(\theta)+t\, \omega_2(\theta)$ and $t\in [0,1]$. Now one has that $\omega(\theta; t)$ is finite  for almost all $\theta \in [\theta_0, \theta_0+2\pi)$ and admits partial derivative with respect to $t$ by construction; moreover, since
$$
\left|\frac{\partial\omega}{\partial t}(\theta; t)\right|\leq \omega_1(\theta)+\omega_2(\theta),
$$
$\mathrm{d}\sigma$ satisfies the hypotheses of Proposition \ref{pmain}. By virtue of Remark \ref{remarkkernel}, we can always construct a POPUC of degree $n$ associated with $\mathrm{d}\sigma$ with a zero at $\theta_{k,1}$. We also see that
$$
\frac{1}{\omega(\theta; t)}\frac{\partial \omega}{\partial t}(\theta; t)=\frac{1}{t}+\frac{\dps \frac{1}{t}}{1-t+\dps \frac{\omega_2(\theta)}{\omega_1(\theta)}\, t}
$$
is a strictly increasing function of $\theta$ on  $[\theta_0, \theta_0+2\pi)$  for each $t\in(0,1)$ . Finally, since $\omega(\theta; 0)=\omega_1(\theta)$ and $\omega(\theta; 1)=\omega_2(\theta)$, the result is a consequence of Theorem \ref{markov}.
\end{proof}

We will use the same arguments, as in the proof of Theorem \ref{markov}, to prove the next theorem, which we will refer as circular Markov theorem for complex conjugate zeros. 

\begin{theorem}\label{markoveven}
 Assume the hypotheses and notation of Theorem \ref{markov} and its proof, except that $P(e^{i\theta_0}; t)=0$. Set $\theta_0=-\pi$. Suppose that $P(\overline{\zeta(t)}; t)=0$ and $\varphi(t)\in (0,\pi)$ $({\rm mod[-\pi, \pi)})$  for each $t \in I_\epsilon(t_0)$. Then $\zeta(t)$ moves strictly counterclockwise along $\mathbb{S}^1$ as $t$ increases on $I_\epsilon(t_0)$, provided that \eqref{mark} is a strictly decreasing function of $\theta$ on  $(-\pi, 0)$ and a strictly increasing function of $\theta$ on $(0, \pi)$.
\end{theorem}
\begin{proof}
Replacing $\zeta(t)$ in \eqref{derivative1} by $\overline{\zeta(t)}$, we get
\begin{align}\label{derivative2}
\overline{\zeta'(t)}=-\dps \frac{\dps \int  \frac{\overline{P_n(e^{i \theta}; t)}}{\overline{e^{i \theta}-\overline{\zeta(t)}}} \dps \frac{\partial P_n}{\partial t}(e^{i\theta}; t)\mathrm{d}\mu(\theta; t)}{\dps \int  \left|\frac{P_n(e^{i \theta}; t)}{e^{i \theta}-\overline{\zeta(t)}}\right|^2\mathrm{d}\mu(\theta; t)}.
\end{align}
Now let $C(t)$ denotes the sum of the denominators of the right hand sides of \eqref{derivative1} and \eqref{derivative2}. Note that $-i\overline{\zeta(t)}\zeta'(t)=\varphi'(t)=i\zeta(t)\overline{\zeta'(t)}$ and 
$$
\frac{i\overline{\zeta(t)}}{z-\overline{\zeta(t)}}-\frac{i\zeta(t)}{z-\zeta(t)}=2 \,\Ima(\zeta(t)) \frac{z}{(z-\overline{\zeta(t)})(z-\zeta(t))}.
$$
If \eqref{derivative1} and \eqref{derivative2} are multiplied by $-i\overline{\zeta(t)}$ and $i\zeta(t)$ respectively and the resulting equations are added, we have
\begin{align}\label{import}
C(t) \varphi'(t)=-2\, \Ima(\zeta(t)) \int \frac{e^{i\theta}}{(e^{i\theta}-\zeta(t))(e^{i\theta}-\overline{\zeta(t)})} P_n(e^{i\theta}; t) \overline{\frac{\partial P_n}{\partial t}(e^{i\theta}; t)}\mathrm{d}\mu(\theta; t).
\end{align}
Replacing $\xi$ in \eqref{last0} by $\overline{\zeta(t)}$, we get
\begin{align*}
&\int \frac{e^{i\theta}}{(e^{i\theta}-\zeta(t))(e^{i\theta}-\overline{\zeta(t)})} P_n(e^{i\theta}; t) \overline{\frac{\partial P_n}{\partial t}(e^{i\theta}; t)}\mathrm{d}\mu(\theta; t)\\[7pt]
\nonumber &=-\int  \frac{e^{i\theta}}{(e^{i\theta}-\zeta(t))(e^{i\theta}-\overline{\zeta(t)})}  |P_n(e^{i\theta}; t)|^2\,\frac{\partial \omega}{\partial t}(\theta; t) \mathrm{d}\mu(\theta).
\end{align*}
Replacing $\xi$ in \eqref{last0} by $\overline{\zeta(t)}$, we obtain
\begin{align}
\label{last}&\int \frac{e^{i\theta}}{(e^{i\theta}-\zeta(t))(e^{i\theta}-\overline{\zeta(t)})} P_n(e^{i\theta}; t) \overline{\frac{\partial P_n}{\partial t}(e^{i\theta}; t)}\mathrm{d}\mu(\theta; t)\\[7pt]
\nonumber &=-\int  \frac{e^{i\theta}}{(e^{i\theta}-\zeta(t))(e^{i\theta}-\overline{\zeta(t)})}  |P_n(e^{i\theta}; t)|^2\,\varpi(\theta; t) \mathrm{d}\mu(\theta; t).
\end{align}
Substituting \eqref{last} into \eqref{import}, we can assert that
\begin{align*}
\varphi'(t)=-2\, \dps \frac{\Ima(\zeta(t))}{C(t)} \int  \frac{e^{i\theta}}{(e^{i\theta}-\zeta(t))(e^{i\theta}-\overline{\zeta(t)})}  |P_n(e^{i\theta}; t)|^2\,\varpi(\theta; t) \mathrm{d}\mu(\theta; t).
\end{align*}
\begin{figure}[H]
\centering
\includegraphics[width=8cm]{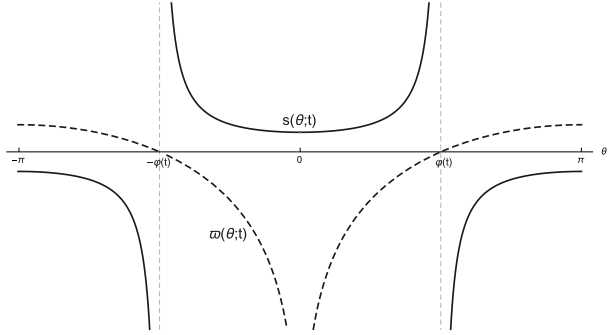}\label{proof2}
\caption{$s$ and $\varpi$ for the circular Markov theorem for complex conjugate zeros.}
\end{figure} 

Observe that, for  each $t \in I_\epsilon(t_0)$, the real-valued function
$$
\frac{e^{i\theta}}{(e^{i\theta}-\zeta(t))(e^{i\theta}-\overline{\zeta(t)})}=\frac{1}{2}\frac{1}{\cos\theta-\cos\varphi(t)},
$$
 is negative for $\theta\in(-\pi, -\varphi(t))\cup (\varphi(t), \pi)$ and positive for $\theta\in (-\varphi(t), \varphi(t))$. Since, for each $t \in I_\epsilon(t_0)$, $\varpi(\theta; t)$ is positive for $\theta\in(-\pi, -\varphi(t))\cup (\varphi(t), \pi)$ and negative for $\theta\in (-\varphi(t), 0) \cup (0, \varphi(t))$, $\varphi'(t)$ is positive (see Figure $3$), which proves the theorem.
\end{proof}

\begin{obs}
If $\mathrm{d}\mu(\theta; t)=-\mathrm{d}\mu(-\theta; t)$ for each $\theta \in [-\pi, \pi)$ $($symmetric measure$)$ and $b=\pm1$, then the nonreal zeros of the corresponding POPUC occur in complex conjugate pairs. Under this more restrictive condition, Theorem \ref{markoveven} `coincides' with Corollary \ref{end}  in case of $\omega(-\theta; t)=\omega(\theta; t)$.
\end{obs}



We can now rephrase Corollary \ref{comp} as follows. 

\begin{coro}\label{coroconjugate}
Let $\mathrm{d}\mu_1(\theta)=\omega_1(\theta)\, \mathrm{d}\mu(\theta)$ and $\mathrm{d}\mu_2(\theta)=\omega_2(\theta)\, \mathrm{d}\mu(\theta)$  be two nonnegative symmetric measures with infinite support on the unit circle parame\-trized by $z=e^{i\theta}$ $(\theta \in  [-\pi, \pi))$ and satisfying the hypotheses of Theorem \ref{markov}. Suppose that $\omega_1(\theta)$ and $\omega_2(\theta)$ are finite, positive and continuous for almost all $\theta$. Let $\omega_2(\theta)/\omega_1(\theta)$ be a strictly decreasing function on  $(-\pi, 0)$ and a strictly increasing function on $(0, \pi)$. Fix $n\geq 2$ and let $\pi> \theta_{1,1}>\cdots >\theta_{n,1}\geq -\pi$ and $\pi> \theta_{1,2}>\cdots >\theta_{n,2}\geq -\pi$ denote the arguments of the zeros of the POPUC of degree $n$ associated with $\mathrm{d}\mu_1$ and $\mathrm{d}\mu_2$, respectively. Then
$$
\theta_{j,1}<\theta_{j,2}
$$ 
for each $j\in \{1, \dots, \floor{n/2}\}$.
\end{coro}
\begin{proof}
We proceed in the same manner as in the proof of Corollary \ref{comp}, but now we construct a POPUC of degree $n$ associated with $\mathrm{d}\sigma$ (defined in Corollary \ref{comp}), $P_n(\cdot; b; \mathrm{d}\sigma)$, whose parameter $b$ is equal to $\pm1$.
\end{proof}

The next proposition is nothing more than a direct consequence of a result by V. B. Lidskii \cite{L55} (see also \cite[Section V.6.]{A64}). 

\begin{proposition}\label{Lidskii}
Let $\mathrm{d}\mu(\theta)$ be a finite nonnegative measure with infinite support on the unit circle parametrized by $z=e^{i\theta}$. Let $b$ be a function of a real variable $t$ defined on a real open interval containing $t_0$ with values in $\mathbb{S}^1$. Assume the existence of the derivative of $b(t)$ near $t=t_0$. Let $P$ be a monic POPUC defined as in \eqref{popuc} for $b=b(t)$. Suppose that $P(\zeta_0; t_0)=0$. Then there exist $\epsilon>0$ and $\delta>0$ such that $C_\delta(\zeta_0) \times I_\epsilon(t_0)$ is in the neighbourhood where $P$ is defined, and there exists $\zeta: I_\epsilon(t_0)\to C_\delta(\zeta_0)$, such that \eqref{zero} holds and, for each $t \in I_\epsilon(t_0)$, $\zeta$ is the unique solution of \eqref{zero} with $\zeta(t)\in C_\delta(\zeta_0)$. Moreover, $\zeta$ is differentiable on $I_\epsilon(t_0)$. Furthermore, $\zeta(t)$ moves strictly counterclockwise along $\mathbb{S}^1$ as $t$ increases on $I_\epsilon(t_0)$, provided that 
\begin{align*}
 i\, \overline{b(t)} b'(t)
 \end{align*} 
 is strictly positive.
\end{proposition}
\begin{proof}
Clearly, the first two statements of the theorem follow as in Theorem \ref{markov}. Throughout the proof, the matrix valued function $\mathrm{G}(t)$ denotes the matrix \eqref{G} for $a_j=a_j(t)$ and $b=b(t)$. In view of the analytic implicit function theorem, we can choose a normalized eigenpair $(\zeta(t), \mathrm{p}(t))$ that depends differentiably on $t$. Write $\zeta(t)=e^{i\varphi(t)}$. Since $\mathrm{G}(t)$ is normal (in particular, unitary), $\zeta(t)=(\mathrm{p}(t), \mathrm{G}(t) \mathrm{p}(t))$. Moreover, since $\zeta(t)$ is a simple eigenvalue of $\mathrm{G}(t)$, $\zeta'(t)=(\mathrm{p}(t), \mathrm{G}'(t) \mathrm{p}(t))$. Thus
\begin{align*}
\varphi'(t)&= \big(\mathrm{p}(t), -i\,\mathrm{G}^*(t) \mathrm{G}'(t)\,\mathrm{p}(t)\big)=\big(\mathrm{p}(t), -i\,\mathrm{G}^*_{n-1}(t) \mathrm{G}'_{n-1}(t)\, \mathrm{p}(t)\big)\\&= i\, \overline{b(t)}\, b'(t) |(\mathrm{p})_{n-1}(t)|^2,
\end{align*}
 $(\mathrm{p})_{n-1}(t)$ being the last component of $\mathrm{p}(t)$. Finally, the result follows because $(\mathrm{p})_{n-1}(t)$ is nonzero\footnote{This is true for any unreduced Hessenberg matrix (cf. \cite[Lemma 2.1]{L95}), although in this case all the components of $\mathrm{p}$ are nonzero (see \cite[Chapter $4$]{S05I}).}.
\end{proof}

In this section we have given readers a taste of the flexibility of our arguments, hopping that they can easily adapt it to a wide variety of situations not considered in this work.

\section{Examples}\label{examples}

In this section, we consider some applications of the results of Section \ref{main} to specific weight functions on the unit circle. The reader should satisfy himself that  the hypotheses of Proposition \ref{pmain} are fulfilled.

\subsection{Degree one Bernstein-Szeg\H{o} polynomials}\label{BS}
Let
$$
\mathrm{d}\mu_\zeta(\theta)=\frac{1-|\zeta|^2}{|1-\zeta e^{i\theta}|^2}\frac{\mathrm{d}\theta}{2\pi}
$$
for $\zeta\in \mathbb{D}$ and $\theta \in [0, 2\pi)$. If one set $\zeta=r e^{i\varphi}$ for $\varphi \in [0, 2\pi)$ and defines the Poisson kernel by
$$
P_r(\theta, \varphi)=\frac{1-r^2}{1+r^2-2r\cos(\theta-\varphi)},
$$
then we may write $\dps\mathrm{d}\mu_\zeta(\theta)=P_r(\theta, -\varphi)\frac{\mathrm{d}\theta}{2\pi}$. The OPUC for this measure are given  by $($cf. \cite[Example 1.6.2]{S05I}$)$ 
$$
Q_n(z; \mathrm{d}\mu_\zeta)=z^n-\overline{\zeta}z^{n-1}\quad (n=1,2,\dots).
$$
An easy computation shows that
\begin{align*}
A(\theta; r, \varphi)=\frac{1}{P_r(\theta, -\varphi)}\frac{\partial P_r}{\partial r}(\theta, -\varphi)&=\frac{4r-2(r^2+1)\cos (\theta+\varphi)}{(r^2-1)(r^2-2r\cos(\theta+\varphi)+1)},\\[7pt]
B(\theta; r, \varphi)=\frac{1}{P_r(\theta, -\varphi)}\frac{\partial P_r}{\partial \varphi}(\theta, -\varphi)&=-\frac{2r\sin(\theta+\varphi)}{r^2-2r\cos(\theta+\varphi)+1}.
\end{align*}
And once we have reached this point, the first thing we must do is to verify if the functions $A$ and $B$  increase $($and decrease$)$ at most once on $(0, 2\pi)$. $($This is a necessary, although not sufficient, condition for a successful use of the results of the previous section.$)$  For illustration, consider the case $A(\theta; 0.1, 0)$ $($solid line$)$, $A(\theta; 0.5, 0)$ $($dash line$)$, $A(\theta; 0.5, \pi)$ $($dotted line$)$ and $A(\theta; 0.1, 3\pi/2)$ $($dash-dotted line$)$, and $B(\theta; 0.1, 0)$ $($solid line$)$, $B(\theta; 0.5, 0)$ $($dash line$)$, $B(\theta; 0.5, \pi)$ $($dotted line$)$ and $B(\theta; 0.1, 3\pi/2)$ $($dash-dotted line$)$ displayed in Figure \ref{F12}. In all these cases, the function $B$  does not fulfil the required conditions.
\begin{figure}[h]
\centering
  \includegraphics[width=5.5cm]{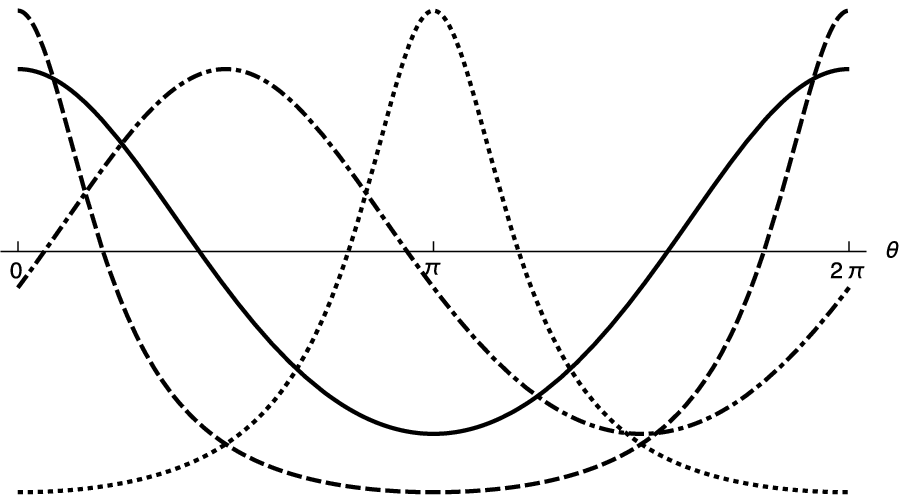}\quad   \includegraphics[width=5.5cm]{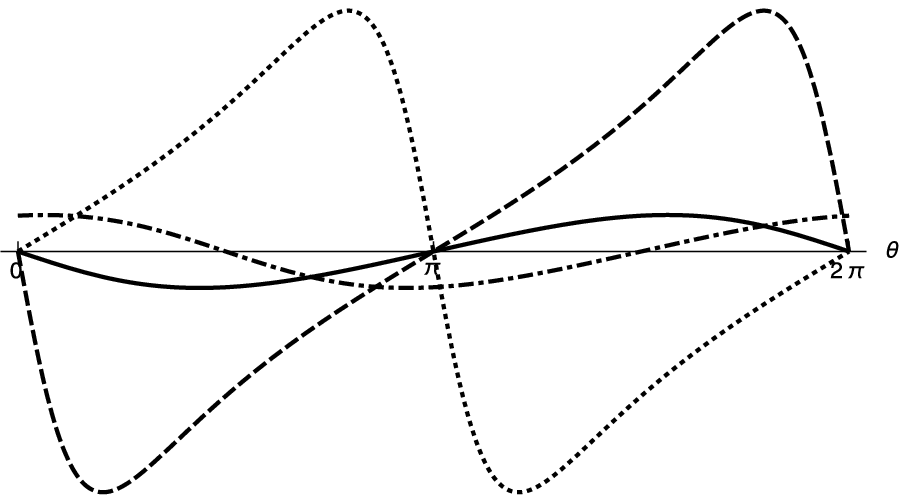}
  \caption{$A(\theta; r, \varphi)$ (left plot) and $B(\theta; r, \varphi)$ (right plot) for certain values of $r$ and $\varphi$.}\label{F12}
\end{figure}

In what follows, for simplicity, we will specialize to the case $r\in (0,1)$ and $\varphi=0$, that is, $\zeta=r$. In this case
$$
A(\theta; r; 0)=\frac{4r-2(r^2+1)\cos\theta}{(r^2-1)(r^2-2r\cos\theta+1)}
$$
 is a strictly decreasing function of $\theta$ on  $(0, \pi)$ and a strictly increasing function of $\theta$ on $(\pi, 2\pi)$. Fix $\xi \in \mathbb{S}^1$. From Remark \ref{remarkkernel}, it may be concluded that the zeros of 
$$
P_{n+1}\left(z; \xi^{n-1} \frac{\xi-r}{\overline{\xi-r}}; \mathrm{d}\mu_\zeta\right)=z^{n+1}-r z^n+r \left(\xi^{n-1} \frac{\xi-r}{\overline{\xi-r}}\right)z-\xi^{n-1} \frac{\xi-r}{\overline{\xi-r}}
$$
are $\xi$ plus the zeros of $K_{n}(\xi, \cdot ; \mathrm{d}\mu_\zeta)$. Hence, the nonreal zeros of $K_{n}(\pm1, \cdot ; \mathrm{d}\mu_\zeta)$ occur in complex conjugate pairs. Thus, by the circular Markov theorem for complex conjugate zeros, these zeros move strictly clockwise on the upper semicircle as $r$ increases on $(0, 1)$.  We can evidently not expect to obtain information about the behavior of the zeros of $K_{n}(\xi, \cdot ; \mathrm{d}\mu_\zeta)$ for each $\xi \in \mathbb{S}^1$. Figure \ref{F22} shows the behaviour of the zeros of $K_{14}(1, \cdot ; \mathrm{d}\mu_\zeta)$ and $K_{14}(i, \cdot ; \mathrm{d}\mu_\zeta)$ for $r=0.1$ (discs), $r=0.5$ (squares) and $r=0.9$ (diamonds). The zeros of $K_{14}(1, \cdot ; \mathrm{d}\mu_\zeta)$ behave exactly as predicted, but on the other hand the zeros of $K_{14}(i, \cdot ; \mathrm{d}\mu_\zeta)$ do not behave in the same way.
\begin{figure}[H]
\centering
  \includegraphics[width=5.5cm]{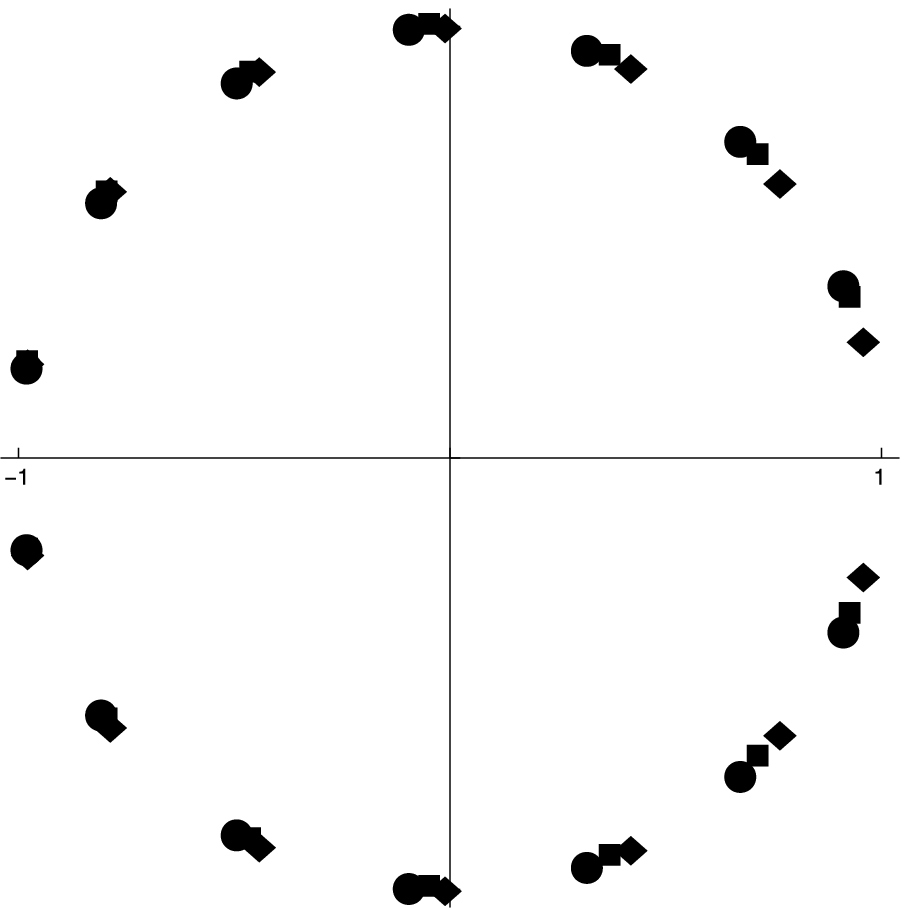}\quad   \includegraphics[width=5.5cm]{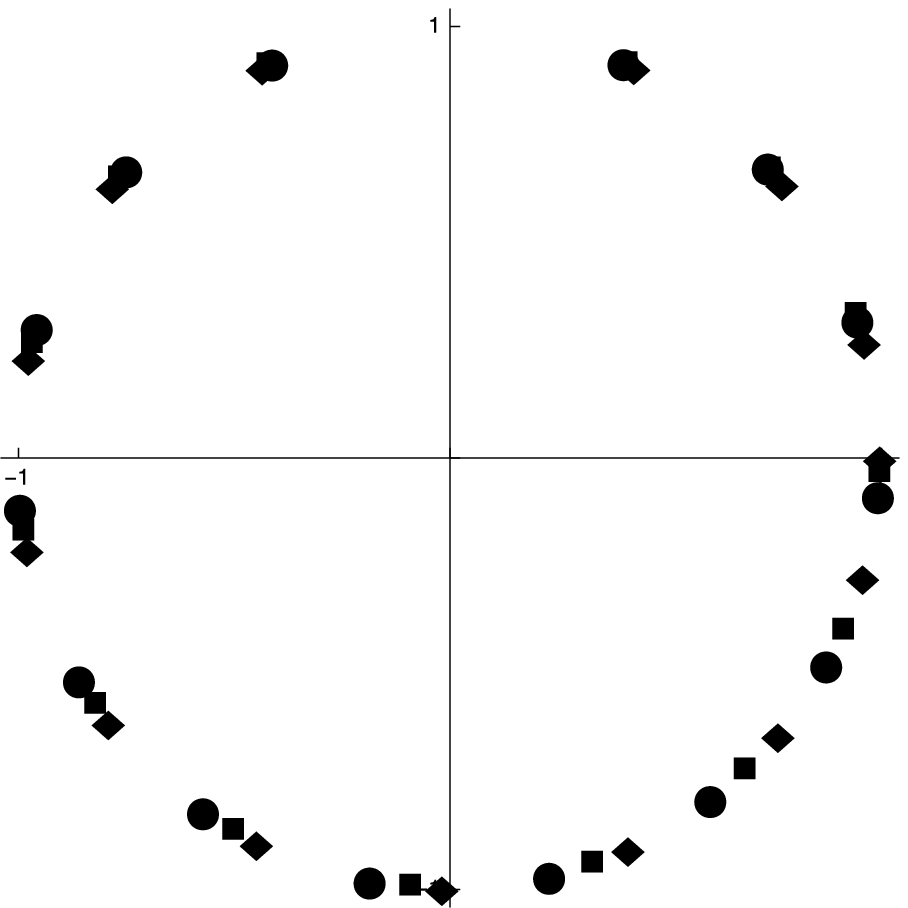}
  \caption{Zeros of $K_{14}(1, \cdot ; \mathrm{d}\mu_\zeta)$ (left plot) and $K_{14}(i, \cdot ; \mathrm{d}\mu_\zeta)$ (right plot) for certain values of $r$.}\label{F22}
\end{figure}

\subsection{Single nontrivial moment}
Let
$$
\mathrm{d}\mu_r(\theta)=(1-r \cos \theta)\frac{\mathrm{d}\theta}{2\pi}
$$
for $r\in(0,1)$ and $\theta \in [0, 2\pi)$. The OPUC for this measure are given by $($see \cite[Example $1.6.4$]{S05I}$)$
$$
Q_n(z; \mathrm{d}\mu_r)=\frac{1}{d_n}\,\sum_{j=0}^n d_j z^j \quad (n=0,1,\dots).
$$
where 
$$
d_j=\frac{d_+^{j+1}-d_-^{j+1}}{d_+-d_-},
$$
and $d_\pm$ are the roots of $r d^2-2d+r=0$, that is,
$$
d_{\pm}=\frac{1}{r} \pm \sqrt{\dps\frac{1}{r^2}-1}.
$$
Now we proceed as in Section \ref{BS}. Indeed, since the function
$$
A(\theta; r)=-\frac{\cos\theta}{1-r \cos\theta}
$$
 is a strictly increasing function of $\theta$ on  $(0, \pi)$ and a strictly decreasing function of $\theta$ on $(\pi, 2\pi)$, by the circular Markov theorem for complex conjugate zeros, we can conclude that the nonreal zeros of 
$$
P_{n+1}\left(z; \pm 1; \mathrm{d}\mu_r\right)=\frac{1}{d_n}\,\sum_{j=0}^{n+1} \left(d_{j-1}\mp d_{n-j}\right)z^j
$$
 move strictly counterclockwise on the upper semicircle as $r$ increases on $(0, 1)$.  By virtue of Corollary \ref{coroconjugate}, we may also compare the zeros of $P_{n+1}\left(\cdot; \pm 1; \mathrm{d}\mu_r\right)$ with those of $P_{n+1}\left(z; \pm 1; \mathrm{d}\mu_\zeta\right)$ for $\varphi=0$. Let $2\pi> \theta_{1}(\mathrm{d}\mu_\zeta)>\cdots >\theta_{n-1}(\mathrm{d}\mu_\zeta)\geq 0$ and $2\pi> \theta_{1}(\mathrm{d}\mu_r)>\cdots >\theta_{n}(\mathrm{d}\mu_r)\geq 0$ denote the arguments of the zeros of $P_{n+1}\left(z; \pm 1; \mathrm{d}\mu_\zeta\right)$ and $P_{n+1}\left(z; \pm 1; \mathrm{d}\mu_r\right)$, respectively. Define $\omega(\theta; r)=(1-r \cos \theta)$. An easy calculation reveals that the function
$$
\frac{P_r(\theta; 0)}{\omega(\theta; r)}=\frac{1-r^2}{(1+r^2-2r\cos\theta)(1-r \cos \theta)}
$$
 is a strictly decreasing function of $\theta$ on  $(0, \pi)$ and a strictly increasing function of $\theta$ on $(\pi, 2\pi)$. Thus, Corollary \ref{comp} implies that
 \begin{align}\label{auxx}
\theta_j(\mathrm{d}\mu_r)< \theta_j(\mathrm{d}\mu_\zeta)
 \end{align}
 when $\theta_j(\mathrm{d}\mu_r) \in (0, \pi)$.  
  Figure \ref{FX} shows the behaviour of the zeros of $K_{14}(1, \cdot ; \mathrm{d}\mu_\zeta)$ (discs) and $K_{14}(1, \cdot ; \mathrm{d}\mu_r)$ (squares) for $r=0.8$ and $P_{15}(\cdot; i; \mathrm{d}\mu_\zeta)$ (discs) and $P_{15}(\cdot; i; \mathrm{d}\mu_r)$ (squares) for $r=0.8$. Although the zeros of $K_{14}(1, \cdot ; \mathrm{d}\mu_\zeta)$ and $K_{14}(1, \cdot ; \mathrm{d}\mu_r)$ behave exactly as predicted, the zeros of $P_{15}(\cdot; i; \mathrm{d}\mu_\zeta)$ and $P_{15}(\cdot; i;$ $\mathrm{d}\mu_r)$, as expected, do not satisfy \eqref{auxx}.
  \begin{figure}[h]
\centering
  \includegraphics[width=5.5cm]{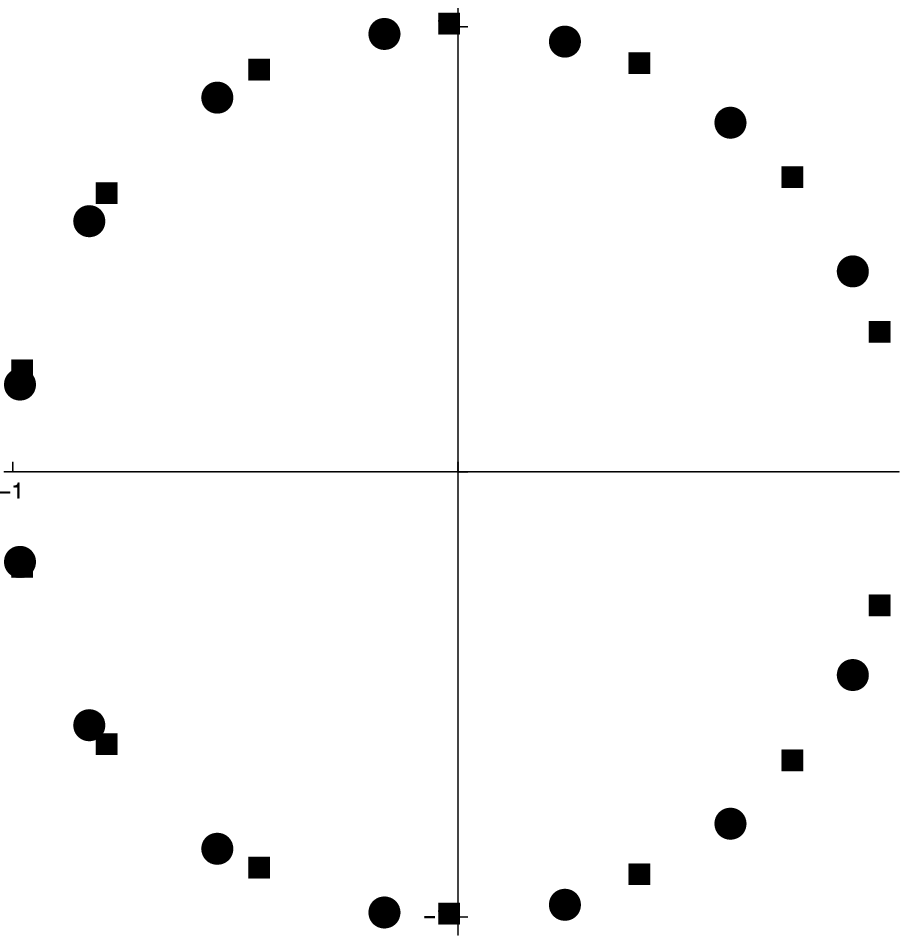}\quad   \includegraphics[width=5.5cm]{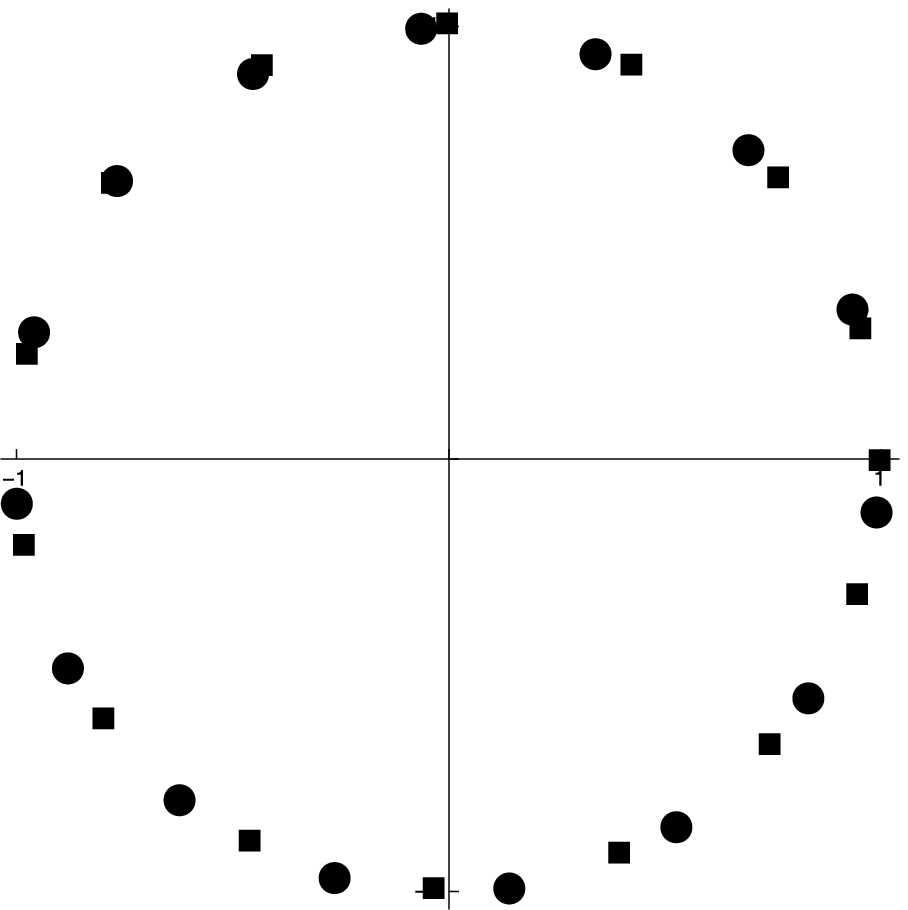}
  \caption{Zeros of $K_{14}(1, \cdot ; \mathrm{d}\mu_\zeta)$ and $K_{14}(1, \cdot ; \mathrm{d}\mu_r)$ (left plot) and $P_{15}(\cdot; i; \mathrm{d}\mu_\zeta)$ and $P_{15}(\cdot; i;$ $\mathrm{d}\mu_r)$ (right plot) for certain values of $r$.}\label{FX}
\end{figure}
  
  \subsection{Jacobi-Szeg\H{o} polynomials}
Let
 \begin{align*}
\mathrm{d}\mu^{(r, s)}(\theta)=\frac{|\Gamma(r+is+1)|^2}{\Gamma(2 r+1)}\,(2-2\cos \theta)^r\dps(-e^{i\theta})^{i s}\, \frac{\mathrm{d}\theta}{2\pi}
\end{align*}
for $r \in (-1/2, \infty)$, $s \in (-\infty, \infty)$, and $\theta \in [-\pi,\pi)$.  There are a variety of specific problems, particularly in statistical physics, which are closely related to this measure. Indeed, $\mathrm{d}\mu^{(r, s)}$ belongs to a class of measures introduced by Fisher and Hartwig in \cite{FH68} which has been the subject of numerous investigations $($see \cite{DIK11} and the references given there$)$.  The following alternative expression for $\theta \in [0, 2\pi)$ is also found in the literature (see \cite{GK83}):
 \begin{align*}
\mathrm{d}\mu^{(r, s)}(\theta)=\frac{|\Gamma(r+is+1)|^2}{\Gamma(2 r+1)} 2^{2 r} e^{(\pi-\theta) s}\left(\sin \dps\frac{\theta}{2} \right)^{2 r}\,\frac{\mathrm{d}\theta}{2\pi}.
\end{align*}
The OPUC for this measure are given by $($see \cite[Sections $1.1$ and $1.2$]{GK83} and \cite[Section $3$]{A85} for more details$)$ 
\begin{align}\label{opuc}
Q_n(z; \mathrm{d}\mu^{(r, s)})=\frac{(2r+1)_n}{(r+i s+1)_n}\, \pFq{2}{1}{-n,r+is+1}{2r+1}{1-z} \quad (n=0,1,\dots).
\end{align}
These polynomials can be expressed in terms of Heisenberg polynomials\footnote{We use the notation, now standard, of \cite{G81, GK83}.}, which live on the Heisenberg group $($see \cite[(1.7)]{GK83}$)$, that is, 
$$
Q_n(e^{i\theta}; \mathrm{d}\mu^{(r, s)})=\frac{n!}{(r+i s+1)_n}\, \dps e^{i \,n\theta/2}\, \dps C_n^{(r-is, r+is+1)}(e^{i\theta/2}).
$$
Define $\omega_1(\theta; r, s)=(2-2\cos \theta)^r (-e^{i\theta})^{i s}$ and $\omega_2(\theta; r, s)=2^{2 r} e^{(\pi-\theta) s}\left(\sin \dps\frac{\theta}{2} \right)^{2 r}$. Hence
\begin{align*}
A(\theta; r, s)=\frac{1}{\omega_1(\theta; r, s)}\frac{\partial \omega_1(\theta; r, s)}{\partial r}&=\log(2-2\cos\theta),\\[5pt]
B(\theta; r, s)=\frac{1}{\omega_2(\theta; r, s)}\frac{\partial \omega_2(\theta; r, s)}{\partial s}&=\pi-\theta.
\end{align*}
 We can therefore apply the results of Section \ref{main} to study the variation of zeros of certain POPUC associated with $\mathrm{d}\mu^{(r, s)}$. 

Given any $\xi=e^{i\theta_0(r,s)}$, we define
$$
b^{(r, s)}(\xi)=\overline{\xi}\, \frac{(r+is+1)_n}{(r-i s+1)_n}\,\frac{\dps \pFq{2}{1}{-n,r-is+1}{2r+1}{1-\xi}}{\dps \pFq{2}{1}{-n,r-is}{2r+1}{1-\xi}}.
$$
By Remark \ref{remarkkernel}, $P_{n+1}\left(z; b^{(r, s)}(\xi); \mathrm{d}\mu^{(r, s)}\right)$ has a zero at $z=\xi$. Assume $\xi=1$ $($or, what is the same, $\theta_0=0$$)$. Since $B(\theta; r, s)$ is a strictly decreasing function of $\theta$ on  $(0, 2\pi)$, by the circular Markov theorem with a fixed zero, the nonreal zeros of $P_{n+1}\left(z; b^{(r, s)}(1); \mathrm{d}\mu^{(r, s)}\right)$ move strictly clockwise along $\mathbb{S}^{1}$ as $s$ increases on $(-\infty, \infty)$. This is the main result of \cite{DR13} (see Theorem $1.2$ therein). Indeed, since
$$
b^{(r, s)}(1)=\frac{(r+i s+1)_{n+1}}{(r-i s+1)_{n+1}},
$$
we may conclude that
\begin{align*}
P_{n+1}\left(z; b^{(r, s)}(1); \mathrm{d}\mu^{(r, s)}\right)=\frac{(2r+2)_n}{(r+is+1)_n}(z-1)\,\pFq{2}{1}{-n,r+is+1}{2r+2}{1-z}.
\end{align*}
Thus, for each $r\in (1/2, \infty)$, the zeros of the polynomial
\begin{align*}
f_n(z; r, s)=\pFq{2}{1}{-n, r+is}{2r}{1-z}
\end{align*}
 move strictly clockwise along $\mathbb{S}^{1}$ as $s$ increases on $(-\infty, \infty)$. But this is also true whenever $r\in (0, 1/2)$ $($see \cite[Example $3.1$]{C19}$)$. Thus warned, the reader should be recall that our conditions are only sufficient.
 
 \begin{obs}\label{remarkPro}
Since
  \begin{align}
\label{Eq0}P_{n+1}\left(z; \dps-\frac{r+i s}{r-i s}\, b^{(r, s)}(1); \mathrm{d}\mu^{(r, s)}\right)&=\frac{(2r)_{n+1}}{(r+is)_{n+1}}\,f_{n+1}(z; r, s)\\[7pt]
\label{Eq1}&=\dps\frac{(n+1)!}{(r+is)_{n+1}}\, e^{i \,(n+1) \theta/2}\, \dps C_{n+1}^{(r-is, r+is)}(e^{i\theta/2}),
\end{align}
whenever $r\in(-1/2,$ $\infty)\setminus\{0\}$ and $s\in(-\infty, \infty)$, it follows $($for example by contradiction and using \cite[(2.5.16)]{AAR99}$)$ that $f_{n+1}(\cdot; r, s)$ and $f_{n+2}(\cdot; r, s)$ are ``consecutive" coprime POPUC; whence \cite[Corollary $2.14.5.$]{S11} shows that their zeros strictly interlace $($in the sense explained in \cite[Definition 1.2]{CP18}$)$ on $\mathbb{S}^1$. This specializes to the result of \cite[Theorem 1.1]{DR13} if $r\in(0, \infty)$. For $r=0$ we have
$$
P_{n+1}\left(z; b^{(0, s)}(1); \mathrm{d}\mu^{(0, s)}\right)=\frac{(n+1)!}{(i s)_{n+1}}(z-1) g_n(z; s),
$$
where
$$
g_n(z; s)=\,\pFq{2}{1}{-n,is+1}{2}{1-z}
$$
and so, by the argument above, it can be also shown that the zeros of $g_{n+1}(\cdot; s)$ and $g_{n+2}(\cdot; s)$ strictly interlace on $\mathbb{S}^1$.
 \end{obs}

As far as we know the dependence of the zeros of $f_n(\cdot; r, s)$ on $r$ has been studied only when $s=0$ $($see \cite[Theorem $2$]{DD00}$)$. However, the case $ \mathrm{d}\mu^{(r)}=\mathrm{d}\mu^{(r,0)}$ $($see \cite[Example 8.2.5]{I05}$)$ is especially simple because there is a direct connection with the ultrashperical polynomials\footnote{This allows us to use an old result due to Stieltjes (see \cite[p. 389]{S87}).}. Indeed, by \eqref{Eq1}, we have
\begin{align*}
f_n(e^{i\theta}; r, 0)=\dps\frac{n!}{(2r)_n}\, e^{i\,n\theta/2}\, \dps C_n^{(r, r)}(e^{i\theta/2})=\dps\frac{n!}{(2r)_n}\, e^{i \,n\theta/2}\, \dps C^{(r)}_n\left(\cos\frac{\theta}{2}\right),
\end{align*}
where $C^{(r)}$ denotes an ultrashperical polynomial (see \cite[$(1.9)$]{GK83}). In any case, since the nonreal zeros of 
$$
P_n(\cdot; -1; \mathrm{d}\mu^{(r)})=\frac{(2r)_n}{(r)_n}\,f_n(\cdot; r, 0)
$$ 
occur in complex conjugate pairs, by the circular Markov theorem for complex conjugate zeros, we can conclude that  the zeros of this polynomial move strictly counterclockwise on the upper semicircle and strictly clockwise on the lower semicircle as $r$ increases on $(-1/2, \infty)$. We now turn to the general case $s\in(-\infty, \infty)$. Since $A(\theta; r, s)=A(-\theta; r, s)$  is a strictly decreasing function of $\theta$ on  $(-\pi, 0)$ and a strictly increasing function of $\theta$ on  $(0, \pi)$ and $\dps(e^{-i\theta})^{i s}\geq (e^{i\theta})^{i s}$ for each  $\theta \in (0, \pi)$ and  $s\in[0, \infty)$,  Corollary \ref{end} implies that for each $s\in[0,\infty)$ the zero of $f_n(\cdot; r, s)$ move strictly counterclockwise on the upper semicircle as $r$ increases on $(1/2, \infty)$. In exactly the same way we may show that for each $s\in(-\infty, 0]$ the zero of $f_n(\cdot; r, s)$ move strictly clockwise on the lower semicircle as $r$ increases on $(1/2, \infty)$. Figure \ref{F33} shows the behaviour of the zeros of $f_{10}(\cdot; r, 1)$ and $f_{10}(\cdot; r, -2)$ for $r=0.1$ (discs), $r=1$ (squares) and $r=17$ (diamonds). Note that the zeros of $f_{10}(\cdot; r, 1)$ whose arguments lie between $0$ and $\pi$ and the zeros of $f_{10}(\cdot; r, -2)$ whose arguments lie between $-\pi$ and $0$ behave exactly as predicted; however, the remaining zeros are not necessarily monotone functions of $r$.
\begin{figure}[H]
\centering
  \includegraphics[width=5.5cm]{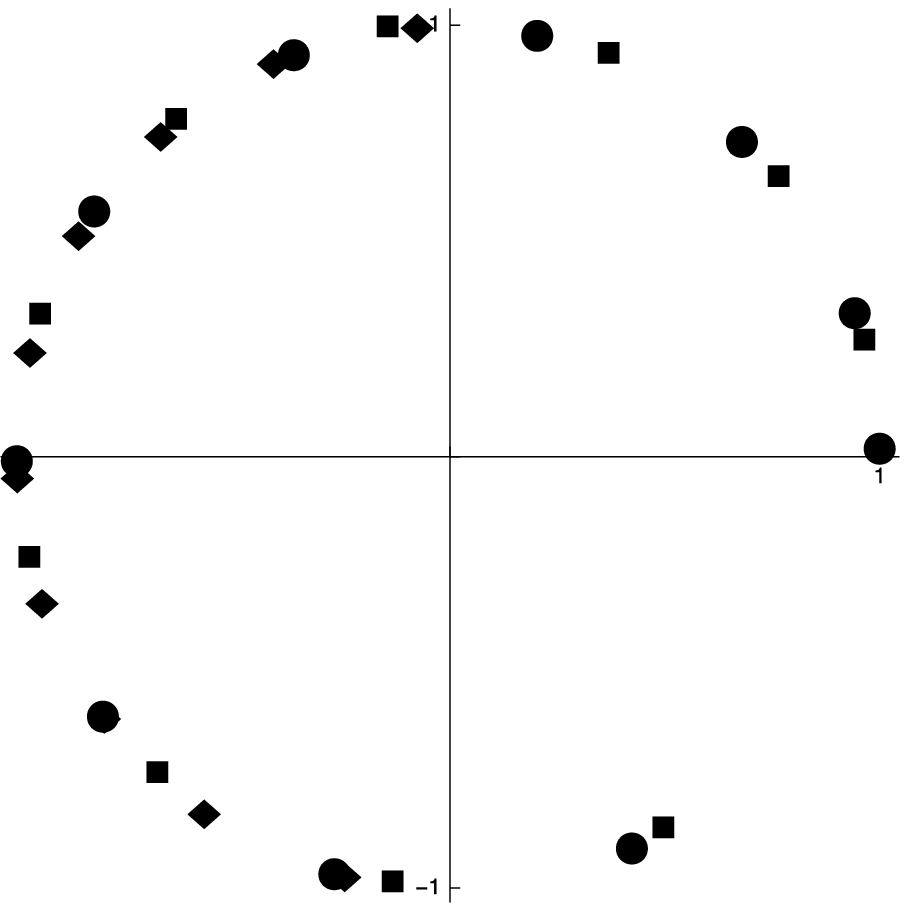}\quad   \includegraphics[width=5.5cm]{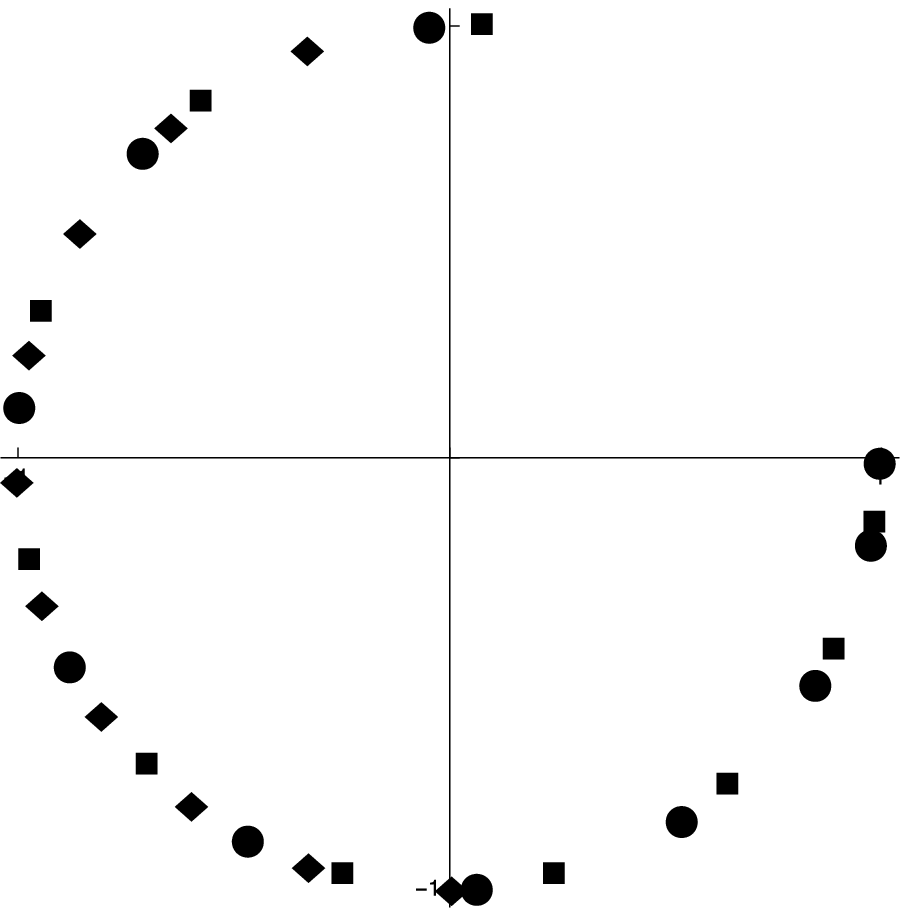}
  \caption{Zeros of $f_{10}(\cdot; r, 1)$ (left plot) and $f_{10}(\cdot; r, -2)$ (right plot) for certain values of $r$.}\label{F33}
\end{figure}

\section*{Acknowledgements}
This work was partially supported by the Centre for Mathematics of the 
University of Coimbra UID/MAT/00324/2019, funded by the Portuguese 
Government through FCT/MEC and co-funded by the European Regional 
Development Fund through the Partnership Agreement PT2020.
\bibliographystyle{plain}      

\bibliography{bib}   

\begin{thebibliography}{10}

\bibitem{AGR88}
G.~Ammar, W.~Gragg, and L.~Reichel.
\newblock Constructing a unitary {H}essenberg matrix from spectral data.
\newblock In {\em Numerical {L}inear {A}lgebra, {D}igital {S}ignal {P}rocessing
  and {P}arallel {A}lgorithms ({L}euven, 1988)}, volume~70 of {\em NATO Adv.
  Sci. Inst. Ser. F Compt. Systems Sci.}, pages 385--395, Berlin, 1991.
  Springer.

\bibitem{GH95}
G.~S. Ammar and C.~He.
\newblock On an inverse eigenvalue problem for unitary {H}essenberg matrices.
\newblock {\em Linear Algebra Appl.}, 15:263--271, 1995.

\bibitem{AAR99}
G.~E. Andrews, R.~Askey, and R.~Roy.
\newblock {\em Special {F}unctions}, volume~71 of {\em Encyclopedia of
  Mathematics and its Applications}.
\newblock Cambridge University Press, Cambridge, 1999.

\bibitem{A85}
R.~Askey.
\newblock Some open problems about special functions and computations.
\newblock In {\em International conference on special functions: theory and
  computation ({T}urin, 1984)}, volume Special Volume, pages 1--22, 1985.

\bibitem{A64}
F.~V. Atkinson.
\newblock {\em Discrete and continuous boundary problems}, volume~8 of {\em
  Mathematics in Science and Engineering}.
\newblock Academic Press, New York-London, 1964.

\bibitem{B93}
B.~Bohnhorst.
\newblock {\em Beitr\"age zur numerischen {B}ehandlung des unit\"aren
  {E}igenwertproblems}.
\newblock PhD thesis, Fakult\"at f\"ur {M}athematik, {U}niversit\"at Bielefeld,
  Bielefeld, Germany, 1993.

\bibitem{BS19}
J.~Breuer and E.~Seelig.
\newblock On the spacing of zeros of paraorthogonal polynomials for singular
  measures.
\newblock {\em arXiv:1908.06737}, 2019.

\bibitem{BH95}
A.~Bunse-Gerstner and C.~He.
\newblock On a {S}turm sequence of polynomials for unitary {H}essenberg
  matrices.
\newblock {\em SIAM J. Matrix Anal. Appl.}, 16:1043--1055, 1995.

\bibitem{C15}
K.~Castillo.
\newblock Monotonicity of zeros for a class of polynomials including
  hypergeometric polynomials.
\newblock {\em Appl. Math. Comput.}, 266:183--193, 2015.

\bibitem{C19}
K.~Castillo.
\newblock On monotonicity of zeros of paraorthogonal polynomials on the unit
  circle.
\newblock {\em Linear Algebra Appl.}, 580:475--490, 2019.

\bibitem{CCP16}
K.~Castillo, R.~Cruz-Barroso, and F.~Perdomo-P{\'\i}o.
\newblock On a spectral theorem in para-orthogonality theory.
\newblock {\em Pacific J. Math.}, 208:71--91, 2016.

\bibitem{CP18}
K.~Castillo and J.~Petronilho.
\newblock Refined interlacing properties for zeros of paraorthogonal
  polynomials on the unit circle.
\newblock {\em Proc. Amer. Math. Soc.}, 146:3285--3294, 2018.

\bibitem{DIK11}
P.~Deift, A.~Its, and I.~Krosovsky.
\newblock Asymptotics of {T}oeplitz, {H}ankel, and {T}oeplitz+{H}ankel
  determinants with {F}isher-{H}artwig singularities.
\newblock {\em Ann. of Math.}, 174:1243--1299, 2011.

\bibitem{DG88}
P.~Delsarte and Y.~Genin.
\newblock The tridiagonal approach to {S}zeg{\H{o}} orthogonal polynomials,
  {T}oeplitz linear systems, and related interpolation problems.
\newblock {\em SIAM J. Math. Anal.}, 19(3):718--735, 1988.

\bibitem{DG91a}
P.~Delsarte and Y.~Genin.
\newblock Tridiagonal approach to the algebraic environment of {T}oeplitz
  matrices. {I}. {B}asic results.
\newblock {\em SIAM J. Matrix Anal. Appl.}, 12(2):220--238, 1991.

\bibitem{DG91b}
P.~Delsarte and Y.~Genin.
\newblock Tridiagonal approach to the algebraic environment of {T}oeplitz
  matrices. {II}. {Z}eros and eigenvalue problems.
\newblock {\em SIAM J. Matrix Anal. Appl.}, 12(3):432--448, 1991.

\bibitem{D70}
J.~Dieudonn{\'e}.
\newblock {\em Treatise on analysis. {V}ol. {II}. {T}ranslated from the
  {F}rench by {I}. {G}. {M}acdonald}, volume 10-{II} of {\em {P}ure and
  {A}pplied {M}athematics}.
\newblock Academic Press, New York-London, 1970.

\bibitem{DR13}
D.~K. Dimitrov and A.~Sri Ranga.
\newblock Zeros of a family of hypergeometric para-orthogonal polynomials on
  the unit circle.
\newblock {\em Math. Nachr.}, 65:41--52, 2013.

\bibitem{DD00}
K.~Driver and P.~Duren.
\newblock Zeros of the hypergeometric polynomials ${F}(-n,b;2b;z)$.
\newblock {\em Indag. Math. N.S.}, 11(1):43--51, 2000.

\bibitem{FH68}
M.~E. Fisher and R.~E. Hartwig.
\newblock Toeplitz determinants. {S}ome applications, theorems and conjectures.
\newblock {\em Adv. Chem. Phys.}, 15:333--353, 1968.

\bibitem{G71}
G.~Freud.
\newblock {\em Orthogonal polynomials}.
\newblock Pergamon Press, Oxford-New York, 1971.

\bibitem{G81}
G.~Gasper.
\newblock Orthogonality of certain functions with respect to complex valued
  weights.
\newblock {\em Can. J. Math.}, XXXIII:1261--1270, 1981.

\bibitem{G46}
Y.~L. Geronimus.
\newblock On the trigonometric moment problem.
\newblock {\em Ann. of Math.}, 47(2):742--761, 1946.

\bibitem{G48}
Ya.~L. Geronimus.
\newblock Polynomials orthogonal on a circle and their applications.
\newblock In {\em Series and {A}pproximation}, volume~3 of {\em 1}, pages
  1--78. Amer. Math. Soc., 1962.

\bibitem{G82o}
W.~B. Gragg.
\newblock {\em Positive definite {T}oeplitz matrices, the {A}rnoldi process for
  isometric operators, and the {G}auss quadrature on the unit circle (in
  Russian)}, pages 16--32.
\newblock Numerical {M}ethods in {L}inear {A}lgebra. Moskov. Gos. Univ.,
  Moscow, 1982.

\bibitem{G86}
W.~B. Gragg.
\newblock The {QR} algorithm for unitary {H}essenberg matrices.
\newblock {\em J. Comp. Appl. Math.}, 16:1--8, 1986.

\bibitem{G82}
W.~B. Gragg.
\newblock Positive definite {T}oeplitz matrices, the {A}rnoldi process for
  isometric operators, and the {G}auss quadrature on the unit circle.
\newblock {\em J. Comp. Appl. Math.}, 46:183--198, 1993.

\bibitem{GR90}
W.~B. Gragg and L.~Reichel.
\newblock A divide and conquer method for unitary and orthogonal eigenproblems.
\newblock {\em Numer. Math.}, 57:695--718, 1990.

\bibitem{GK83}
P.~C. Greiner and T.~H. Koornwinder.
\newblock Variations on the {H}eisenberg spherical harmonics.
\newblock {\em Report ZW 186/83, Mathematisch Centrum, Amsterdam}, 1983.

\bibitem{I05}
M.~E.~H. Ismail.
\newblock {\em Classical and quantum orthogonal polynomials in one variable},
  volume~98 of {\em Encyclopedia of Mathematics and Its Applications}.
\newblock Cambridge University Press, Cambridge, 2005.

\bibitem{JNT89}
W.~B. Jones, O.~Nj{\aa}stad, and W.~J. Thron.
\newblock Moment theory, orthogonal polynomials, quadrature, and continued
  fractions associated with the unit circle.
\newblock {\em Bull. London Math. Soc.}, 21:113--152, 1989.

\bibitem{KN07}
R.~Killip and I.~Nenciu.
\newblock {CMV}: The unitary analogue of {J}acobi matrices.
\newblock {\em Comm. Pure Appl. Math.}, LX:1148--1188, 2007.

\bibitem{KP87}
A.~Kro{\'o} and F.~Peherstorfer.
\newblock On the zeros of polynomials of minimal ${L}_p$ norm.
\newblock {\em Proc. Amer. Math. Soc.}, 101:652--656, 1987.

\bibitem{L95}
R.~B. Lehoucq.
\newblock {\em Analysis and implementation of an implicitly restarted {A}rnoldi
  iteration}.
\newblock PhD thesis, Rice University, Houston, Texas, 1995.

\bibitem{L55}
V.~B. Lidskii.
\newblock Oscillation theorems for a canonical system of differential equations
  ({I}n {R}ussian).
\newblock {\em Dokl. Akad. Nauk SSSR (N.S.)}, 102:877--880, 1955.

\bibitem{L19}
Y.~C. Lun.
\newblock On zeros of paraorthogonal polynomials.
\newblock {\em Proc. Amer. Math. Soc.}, 8:3389--3399, 2019.

\bibitem{M86}
A.~Markoff.
\newblock Sur les racines de certaines \'equations (second note).
\newblock {\em Math. Ann.}, 27:177--182, 1886.

\bibitem{MSS19}
A.~Mart{\'\i}nez-Finkelshtein, B.~Simanek, and B.~Simon.
\newblock Poncelet's theorem, paraorthogonal polynomials and the numerical
  range of compressed multiplication operators.
\newblock {\em Adv. Math.}, 349:992--1035, 2019.

\bibitem{S12}
B.~Simanek.
\newblock Zeros of non-{B}axter paraorthogonal polynomials on the unit circle.
\newblock {\em Constr. Approx.}, 35:107--121, 2012.

\bibitem{S16}
B.~Simanek.
\newblock An electrostatic interpretation of the zeros of paraorthogonal
  polynomials on the unit circle.
\newblock {\em SIAM J. Math. Anal.}, 48:2250--2268, 2016.

\bibitem{S19}
B.~Simanek.
\newblock Zero spacings of paraorthogonal polynomials on the unit circle.
\newblock {\em arXiv:1907.01604}, 2019.

\bibitem{S05I}
B.~Simon.
\newblock {\em Orthogonal polynomials on the unit circle. {P}art {I}.
  {C}lassical {T}heory}, volume~54 of {\em Amer. Math. Soc. Coll. Publ.}
\newblock Amer. Math. Soc., Providence, RI, 2005.

\bibitem{S07b}
B.~Simon.
\newblock {CMV} matrices: {F}ive years after.
\newblock {\em J. Comp. Appl. Math.}, 208:120--154, 2007.

\bibitem{S07}
B.~Simon.
\newblock Rank one perturbations and the zeros of paraorthogonal polynomials on
  the unit circle.
\newblock {\em J. Math. Anal. Appl.}, 329:376--382, 2007.

\bibitem{S11}
B.~Simon.
\newblock {\em Szeg\H{o}'s theorem and its descendants: {S}pectral theory for
  $L^2$ perturbations of orthogonal polynomials}.
\newblock M. B. Porter Lectures. Princeton University Press, Princeton, 2011.

\bibitem{Si15}
B.~Simon.
\newblock {\em Basic {C}omplex {A}nalysis. {A} {C}omprehensive {C}ourse in
  {A}nalysis, {P}art {2A}}.
\newblock American Mathematical Society, Providence, RI, 2015.

\bibitem{S87}
T.~J. Stieltjes.
\newblock Sur les racines de l'equation ${X}_n=0$.
\newblock {\em Acta Math.}, 9:385--400, 1887.

\bibitem{S75}
G.~Szeg\H{o}.
\newblock {\em Orthogonal polynomials}, volume~23.
\newblock Amer. Math. Soc. Coll. Publ., Amer. Math. Soc., Providence, R. I.,
  4th edition, 1975 edition, 1939.

\bibitem{W93}
D.~S. Watkins.
\newblock Some perspectives on the eigenvalue problem.
\newblock {\em SIAM Rev.}, 35(3):430--471, 1993.

\bibitem{W07}
M.~L. Wong.
\newblock First and second kind paraorthogonal polynomials and their zeros.
\newblock {\em J. Approx. Theory}, 146:282--293, 2007.

\end{thebibliography}

\end{document}